\documentclass{article}

\usepackage{arxiv}

\usepackage{placeins}
\usepackage{enumitem}
\usepackage{mathtools}
\usepackage{hyperref}
\usepackage{doi}
\usepackage{amsfonts,amsmath,amssymb,amsopn,amsthm}
\usepackage{graphicx}
\newcommand{\Rz}{\mathbb{R}}
\newcommand{\Rzsp}{\mathbb{R}_{>0}}
\newcommand{\Kn}{K\!n}
\newcommand{\Ma}{M\!a}
\newcommand{\Rey}{R\!e}

\newcommand{\matD}{\frac{\mathrm{D}}{\mathrm{D} t}}
\newcommand{\matDil}{\mathrm{D}/(\mathrm{D} t)}
\usepackage{stmaryrd}
\newcommand{\bm}{\boldsymbol}
\allowdisplaybreaks

\usepackage[numbers,sort&compress]{natbib}

\newtheorem{theorem}{Theorem}

\newtheorem{proposition}{Proposition}
\newtheorem{lemma}{Lemma}

\theoremstyle{definition}
\newtheorem{definition}{Definition}

\newtheorem{remark}{Remark}

\title{Limit consistency of lattice Boltzmann equations}

\author{
    \href{https://orcid.org/0000-0001-8555-4245}{\includegraphics[scale=0.06]{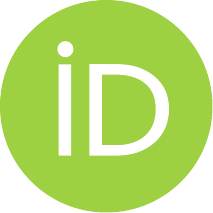}\hspace{1mm}Stephan Simonis}\thanks{Corresponding author}\\
	Institute for Applied and Numerical Mathematics\\
    Karlsruhe Institute of Technology\\
    76131 Karlsruhe, Germany \\
	\texttt{stephan.simonis@kit.edu} \\
	\And
    \href{https://orcid.org/0000-0003-1026-6462}{\includegraphics[scale=0.06]{orcid.pdf}\hspace{1mm}Mathias J. Krause}\\
    Lattice Boltzmann Research Group \\
    Karlsruhe Institute of Technology \\
    76131 Karlsruhe, Germany \\
}

\renewcommand{\headeright}{S.\ Simonis and M.J.\ Krause}
\renewcommand{\shorttitle}{Limit consistency of lattice Boltzmann equations}

\hypersetup{
pdftitle={LCofLBE},
pdfsubject={},
pdfauthor={Stephan Simonis}
}

\begin{document}
\maketitle

\begin{abstract}
We establish the notion of limit consistency as a modular part in proving the consistency of lattice Boltzmann equations (LBEs) with respect to a given partial differential equation (PDE) system. 
The incompressible Navier--Stokes equations (NSE) are used as paragon. 
Based upon the hydrodynamic limit of the Bhatnagar--Gross--Krook (BGK) Boltzmann equation towards the NSE, we provide a successive discretization by nesting conventional Taylor expansions and finite differences. 
We track the discretization state of the domain for the particle distribution functions and measure truncation errors at all levels within the derivation procedure. 
Via parametrizing equations and proving the limit consistency of the respective families of equations, we retain the path towards the targeted PDE at each step of discretization, i.e.\ for the discrete velocity BGK Boltzmann equations and the space-time discretized LBEs. 
As a direct result, we unfold the discretization technique of lattice Boltzmann methods as chaining finite differences and provide a generic top-down derivation of the numerical scheme which upholds the continuous limit.
\end{abstract}



\section{Introduction}

Due to the interlacing of discretization and relaxation, the lattice Boltzmann method (LBM) renders distinct advantages in terms of parallelizability \cite{wolf2000lattice,krause2010fluid}. 
Primarily out of this reason, meanwhile the LBM has become an established alternative to conventional approximation tools for the incompressible Navier--Stokes equations (NSE) \cite{lallemand2021lattice} where optimized scalability to high performance computers (HPC) is crucial. 
As such, LBM has been used with various extensions for applicative scenarios  \cite{krause2021openlb,
kummerlander2022olb15,
simonis2022forschungsnahe,
mink2021comprehensive,
dapelo2021lattice-boltzmann,
haussmann2021fluid-structure}, such as transient computer simulations of turbulent fluid flow with the help of assistive numerical diffusion \cite{haussmann2019direct,simonis2021linear} or large eddy simulation in space \cite{siodlaczek2021numerical} and in time \cite{simonis2022temporal}. 
Nonetheless, the LBM's relaxation principle does come at the price of inducing a bottom-up method, which stands in contrast to conventional top-down discretization techniques such as finite difference methods. 
This intrinsic feature complicates the rigorous numerical analysis of LBM. 
Although several contributions towards this aim exist (see for example \cite{elton1995convergence,
he1997theory,
lallemand2000theory,
junk2000discretizations,
junk2005asymptotic,
dubois2008equivalent,
banda2008lattice,
junk2009weighted,
ubertini2010three,
dellar2013interpretation} and recently \cite{masset2020linear,wissocq2022hydrodynamic}), limitations to specific target equations or LBM formulations persist and thus, to the knowledge of the authors, no generic theory has been established. 
In addition, non-predictable stability and thus convergence limitations for multi-relaxation-time LBM have been indicated by exploratory computing \cite{simonis2021linear} with the help of the highly parallel C++ data structure OpenLB \cite{krause2021openlb}. 
As a first step towards an analytical toolbox which is both, rigorous and modular, the recent works \cite{simonis2020relaxation,simonis2022constructing} shift the perspective on LBM from being exclusively suited for computational fluid dynamics towards a standalone numerical method for approximating partial differential equations (PDE) in general.
Therein, the authors propose an inverse path from the target PDE towards top-down constructed moment systems with generalized Maxwellians which partly form the basis of LBM. 
Here, we continue this path by top-down discretization on condition of limit preservation.

To that end, the present work establishes the abstract concept of limit consistency for the purpose of proving that the LBM-discretization retains an \textit{a priori} given continuous limit. 
In particular, we recall and elaborate the initial idea of Krause \cite{krause2010fluid}, where the classical consistency notion is modulated to analyze truncation errors with respect to families of equations. 
Based upon the proven diffusive limit \cite{saint-raymond2003bgk} of the parametrized Bhatnagar--Gross--Krook Boltzmann equation (BGKBE) towards the NSE, we provide a successive discretization by nesting conventional Taylor expansions and finite differences. 
We measure truncation errors at all levels within the derivation procedure, via keeping track of the discretization state of the domain for the particle distribution functions. 
Parametrizing the equations and proving the limit consistency of the respective families, we uphold the path towards the targeted PDE at each step of discretization, i.e.\ for the discrete velocity BGKBE and the space-time discretized lattice Boltzmann equation (LBE). 
As a direct result, we unfold the discretization technique of lattice Boltzmann methods as chaining finite differences and provide a generic top-down derivation of the numerical scheme. 
Validating our approach, the finite difference interpretation of LBM matches the previous \cite{junk2001finite} and current rigorous observations \cite{bellotti2022finite,bellotti2022rigorous} in the literature. 
Notably, the presented approach has been applied successfully to derive limit consistent LBMs for advection--diffusion equations and homogenized Navier--Stokes equations \cite{simonis2023pde}.

The rest of this paper is structured as follows. 
In Section~\ref{sec:continuousEquations}, we introduce the continuous equations on mesoscopic and macroscopic levels, and recall the formality of the passage from one to the other. 
In Section~\ref{sec:limitConsistency}, the notion of limit consistency is defined. 
Consequently, we discretize the mesoscopic equations and prove that the diffusive limit is consistently upheld in each level of discretization. 
Section~\ref{sec:conclusion} summarizes and assesses the presented results and suggests future studies. 

\section{Continuous mathematical models}\label{sec:continuousEquations}

\subsection{Targeted partial differential equation system}
We aim to approximate the \(d\)-dimensional incompressible Navier--Stokes equations (NSE) as an initial value problem (IVP)
\begin{align}\label{eq:incNSE}
\begin{cases} 
	\bm{\nabla} \cdot \bm{u} = 0  	 \quad &\text{in }  I \times \Omega , \\
	\partial_{t} \bm{u} + \frac{1}{\rho}\bm{\nabla} p + \bm{\nabla} \cdot \left( \bm{u} \otimes \bm{u} \right)  - \nu \bm{\Delta }\bm{u} = \bm{F} 				 \quad &\text{in }  I \times \Omega , \\
	\bm{u} \left(  0, \cdot \right) =  \bm{u}_{0}   \quad &\text{in } \Omega, 
\end{cases}
\end{align}
where the spatial domain \(\Omega \subseteq \mathbb{R}^{d}\) is periodically embedded in reals, \(I \subseteq \mathbb{R}_{\geq 0}\) denotes the time interval, and \((t, \bm{x}) \in I \times \Omega\). 
Further, \(\bm{u}\colon I\times \Omega \to \mathbb{R}^{3}\) is the velocity field with a smooth initial value \(\bm{u}_{0}\), \(p \colon I \times \Omega \to \mathbb{R} \) is the scalar pressure, \(\nu > 0 \) prescribes a given viscosity, \(\rho\colon I\times \Omega \to \mathbb{R}\) is an assumingly constant density, and \(\bm{F}\colon I\times \Omega \to \mathbb{R}^{d}\) denotes a given force field. 
Unless stated otherwise, \(\bm{\nabla}\) is the gradient operator and \(\bm{\Delta}\) the Laplace operator, respectively in space, and partial derivatives with respect to \(\cdot\) are abbreviated with \(\partial_{\cdot} \coloneqq \partial / (\partial\cdot)\).

\subsection{Kinetic equation system}

To render this work self-standing, we first recall the results of Krause \cite{krause2010fluid} and Saint-Raymond \cite{saint-raymond2003bgk} at the continuous level and proceed in Section~\ref{sec:limitConsistency} with the successive discretization down to the LBE. 
At all discretization levels, we reassure that the continuous limiting system \eqref{eq:incNSE} is sustained up to the desired order of magnitude in the parameter \(h > 0\) which is to be affiliated with the discretization. 

Let the domain $\Omega \subseteq \mathbb{R}^d$, with dimension $d=3$ unless stated otherwise, frame a large number of particles interacting through the compound of a rarefied gas. 
Equalizing the mass $m \in \mathbb{R}_{>0}$, the particles are interpreted as molecules.  
	\begin{definition}
The state of a one-particle system is assumed to depend on position $\bm{x} \in \Omega$ and velocity $\bm{v} \in \Xi$ at time $t \in I=[t_0,t_1] \subseteq \mathbb{R}$ with \(t_{1}>t_{0}>0\), where $\Omega \subseteq \mathbb{R}^d$ denotes the positional space, $\Xi = \mathbb{R}^d$ is the velocity space, \(\mathfrak{P} \coloneqq \Omega \times \Xi\) is the phase space, and the crossing \(\mathfrak{R} \coloneqq I \times \Omega \times \Xi\) defines the time-phase tuple. 
The probability density function 
\begin{align} \label{eq: boltz statistic f}
 f\colon\mathfrak{R} \to \mathbb{R}_{>0} , (t,\bm{x},\bm{v}) \mapsto f(t,\bm{x} , \bm{v} ) 
\end{align}
tracks the dynamics of the particle distribution and hence defines the state of the system.
The thus evoked evolution of probability density is reigned by the Boltzmann equation (BE)
\begin{align} \label{eq: boltz statistic equation 2}
  \left( \partial_{t} + \bm{v} \!\cdot\! \bm{\nabla}_{\bm{x}} + \frac{\bm{F}}{m} \!\cdot\! \bm{\nabla}_{\bm{v}} \right) f = J(f,f ), 
\end{align}
where \(f (0, \cdot, \cdot ) = f_{0}  \) supplements a suitable initial condition. 
The operator 
\begin{align}
J\left( f, f\right) =  \int_{\mathbb{R}^{3}} \int_{S^{2}} 
\vert \bm{v} - \bm{w} \vert 
\left[ 
f \left( t, \bm{x}, \bm{v}^{\prime} \right) 
f \left( t, \bm{x}, \bm{w}^{\prime} \right) 
- 
f \left( t, \bm{x}, \bm{v} \right) 
f \left( t, \bm{x}, \bm{w} \right) 
\right]
\,\mathrm{d}\bm{N} 
\,\mathrm{d}\bm{w}
\end{align}
models the collision, where \(\mathrm{d}\bm{N}\) is the normalized surface integral with the unit vector \(\bm{N}\in S^{2}\) and \(\left( \bm{v}^{\prime}, \bm{w}^{\prime} \right)^{\mathrm{T}} = T_{\bm{N}}\left( \bm{v}, \bm{w} \right)^{\mathrm{T}}\) result from the transformation \(T_{\bm{N}}\) that models hard sphere collision \cite{babovsky1998boltzmann}. 
\end{definition}
The following quantities are moments of $f$ created through integration over velocity. 
\begin{definition}\label{def:moments}
Let \(f\) be given in the sense of \eqref{eq: boltz statistic f}. 
Then we define the moments
\begin{align}
n_f &\colon\begin{cases}
I\times\Omega\to\mathbb{R}_{>0},  \\
(t,\bm{x})\mapsto n_f(t,\bm{x}) \coloneqq \int\limits_{\mathbb{R}^d} f(t,\bm{x},v) \,\mathrm{d}\bm{v}, 
\end{cases} \label{eq: boltz statistic f particle density} \\
\rho_f &\colon\begin{cases} 
I\times\Omega\to\mathbb{R}_{>0}, \\
(t,\bm{x})\mapsto \rho_f(t,\bm{x}) \coloneqq mn_f(t,\bm{x}), 
\end{cases} \label{eq: boltz statistic f mass density} \\
\bm{u}_f &\colon\begin{cases} 
I\times\Omega\to\mathbb{R}^{d}, \\
(t,\bm{x})\mapsto u_f(t,\bm{x}) \coloneqq \frac{1}{n_f(t,\bm{x})} \int\limits_{\mathbb{R}^d} \bm{v} f(t,\bm{x},\bm{v}) \,\mathrm{d}\bm{v}, 
\end{cases} \label{eq: boltz statistic f velocity} \\
\mathbf{P}_f &\colon\begin{cases} 
I\times\Omega\to\mathbb{R}^{d\times d},  \\
(t,\bm{x}) \mapsto \mathbf{P}_f(t,\bm{x}) \coloneqq m\int\limits_{\mathbb{R}^d} \left[\bm{v} - \bm{u}_f(t,\bm{x})\right] \otimes \left[\bm{v} - \bm{u}_f(t,\bm{x})\right] f(t,\bm{x},\bm{v})\,\mathrm{d}\bm{v}, 
\end{cases} \label{eq: boltz statistic f stress} \\
p_f &\colon \begin{cases} 
I\times\Omega \to \mathbb{R}_{>0}, \\
(t,\bm{x}) \mapsto p_f(t,\bm{x}) \coloneqq \frac{1}{d} \sum\limits^d_{i=1} \left(\mathbf{P}_f\right)_{i,i}(t,\bm{x}), 
\end{cases}
\label{eq: boltz statistic f pressure} 
\end{align}
respectively as particle density, mass density, velocity, stress tensor, and pressure. 
Here and in the following, the moments of $f$ are indexed correspondingly with \(\cdot_{f}\). 
\end{definition}
Notably, the absolute temperature $T$ is determined implicitly by an ideal gas assumption \begin{align}\label{eq:ideal gas law}
p_f=n_f RT,
\end{align}
where $R>0$ is the universal gas constant. 
To a dedicated order of magnitude in characteristic scales, the above moments approximate the macroscopic quantities conserved by the incompressible NSE \cite{gorban2018hilbert}. 
Equilibrium states $f^{\mathrm{eq}}$ such that 
\begin{align} \label{eq: boltz statistic equilibrium 1}
  J(f^{\mathrm{eq}}, f^{\mathrm{eq}}) =0 \quad \text{in }I\times\Omega,
\end{align}
exist \cite{gorban2018hilbert}. 
Via the gas constant \(R = k_{\mathrm{B}}/m \in\mathbb{R}_{>0}\) (where \(k_{\mathrm{B}}\) is the Boltzmann constant, and \(m\) is the particle mass), and $T\in\mathbb{R}_{>0}$, $n_f$ and $u_f$, the equilibrium state is found to be of Maxwellian form
\begin{align} \label{eq: boltz statistic equilibrium Max}
	f^{\mathrm{eq}}(t,\bm{x},\bm{v}) \colon 
	\mathfrak{R} \to \mathbb{R}, 
	(t,\bm{x},\bm{v}) \mapsto \frac{n_f(t,\bm{x})}{\left(2 \pi R T \right)^{\frac{d}{2}}} \exp\left(-\frac{\left[ \bm{v} - \bm{u}_f(t,\bm{x}) \right]^2}{2 R T}\right) .
\end{align}
\begin{remark}
We identify $f^{\mathrm{eq}}/n_f$ as $d$-dimensional normal distribution for $\bm{v}\in\mathbb{R}^d$ with expectation $\bm{u}_f$ and covariance $RT \mathbf{I}_d$. 
In this regard, the arguments of \(f^{\mathrm{eq}}\) regularly appear in terms of moments \(f^{\mathrm{eq}} ( n_{f}, \bm{u}_{f}, T)\) (e.g.\ \cite{krause2010fluid,he1997theory,junk2005asymptotic,lallemand2000theory}). 
\end{remark}
From $f^{\mathrm{eq}}/n_f$ being a density function, we find
\begin{align} \label{eq: boltz statistic equilibrium 7}
  \rho_{f^{\mathrm{eq}}} \stackrel{\eqref{eq: boltz statistic f mass density}}{=} m \int\limits_{\mathbb{R}^d} f^{\mathrm{eq}} (t,\bm{x},\bm{v})\,\mathrm{d}\bm{v}~= mn_f =\rho_f .
\end{align}
Thus
\begin{align} \label{eq: boltz statistic equilibrium 8}
  \bm{u}_{f^{\mathrm{eq}}}\stackrel{\eqref{eq: boltz statistic f velocity}}{=} \frac{1}{n_{f^{\mathrm{eq}}}} \int\limits_{\mathbb{R}^d} \bm{v} f^{\mathrm{eq}} (t,\bm{x},\bm{v})\,\mathrm{d}\bm{v}~= \bm{u}_f .
\end{align}
The covariance matrix of $f^{\mathrm{eq}}/n_f$ for a perfect gas \eqref{eq:ideal gas law}, verifies the conservation of pressure 
\begin{align} \label{eq: boltz statistic equilibrium 10}
  p_{f^{\mathrm{eq}}} & \stackrel{\eqref{eq: boltz statistic f pressure}}{=}  \frac{1}{d}m \int\limits_{\mathbb{R}^d} \left(\bm{v} - \bm{u}_{f^{\mathrm{eq}}}\right)^{2} f^{\mathrm{eq}} (t,\bm{x},\bm{v})\,\mathrm{d} \bm{v}  \nonumber\\
	& =  \frac{1}{d}m \int\limits_{\mathbb{R}^d} \left(\bm{v} - \bm{u}_{f}\right)^{2} f^{\mathrm{eq}} (t,\bm{x},\bm{v})\,\mathrm{d}\bm{v}  \nonumber\\
	& =  \frac{1}{d} m n_f \sum_{i=1}^{d} RT  \nonumber\\
	& = p_f. 
\end{align}
\begin{definition}
According to Bhatnagar, Gross and Krook (BGK) \cite{bhatnagar1954model} we can simplify the collision operator $J$ in \eqref{eq: boltz statistic equation 2} to  
\begin{align} \label{eq:collisionOperatorQ}
  Q(f) \coloneqq -\frac{1}{\tau}(f-M_{f}^{\mathrm{eq}}) \quad & \text{in } \mathfrak{R} ,  
\end{align}
where $\tau$ denotes the relaxation time between collisions, and \(M^{\mathrm{eq}}_{f} \coloneqq f^{\mathrm{eq}}(t,\bm{x},\bm{v})\) is a formal particular Maxwellian determined by the zeroth and first order moments of \(f\), \(n_{f}\) and \(\bm{u}_{f}\), respectively.
\end{definition}
\begin{remark}
The conservation of both, \(\rho_{f}\) and \(\bm{u}_{f}\), respectively \eqref{eq: boltz statistic equilibrium 7} and \eqref{eq: boltz statistic equilibrium 8}, is upheld, since \(\ln(M^{\mathrm{eq}}_{f})\) is a collision invariant of \(Q\) (cf. \cite[Theorem 1.5]{krause2010fluid}).
\end{remark}
\begin{definition}
With $Q$ from \eqref{eq:collisionOperatorQ} implanted in \eqref{eq: boltz statistic equation 2}, the BGK Boltzmann equation (BGKBE) reads
\begin{align}\label{eq:BGKBE}
  \underbrace{\left( \partial_t + \bm{v} \!\cdot\! \bm{\nabla}_{\bm{x}} + \frac{\bm{F}}{m} \!\cdot\! \bm{\nabla}_{\bm{x}}\right) }_{\eqqcolon~\frac{\mathrm{D}}{\mathrm{D}t}}  f = Q(f) \quad & \text{in } \mathfrak{R} , 
\end{align}
where \(\mathrm{D}/(\mathrm{D} t)\) is referred to as material derivative, and \( f (0, \cdot , \cdot  ) = f_{0} \) sets a suitable initial condition. 
\end{definition}
Here and below, the variable \(f\) is renamed to obey \eqref{eq:BGKBE} instead of \eqref{eq: boltz statistic equation 2}. 
\begin{remark}
The global existence of solutions to the BGKBE \eqref{eq:BGKBE} has been rigorously proven in \cite{perthame1989global}. 
Weighted \(L^{\infty}\) bounds and uniqueness have later been established on bounded domains \cite{perthame1993weighted} and in \(\mathbb{R}^{d}\) \cite{mischler1996uniqueness}.  
\end{remark}

\subsection{Diffusive limit}

We connect the BGKBE \eqref{eq:BGKBE} to the NSE \eqref{eq:incNSE} via diffusive limiting. 
To this end, a formal verification of the continuum balance equations \eqref{eq:incNSE} for the moments \(\rho_{f}\) and \(\bm{u}_{f}\) in Definition~\ref{def:moments} is conducted. 
Parts of the following derivation are taken from \cite{krause2010fluid}.
The limiting is done in three steps. 
As suggested in \cite{simonis2023pde}, we stress that neither the derivation, nor the references follow the aim of completeness, but rather are meant to illustrate the scale-bridging from the BGKBE toward the incompressible NSE only. 
The limiting is done in three steps (see e.g.\ \cite{krause2010fluid}). 

\subsubsection{Step 1: Mass conservation and momentum balance} 
Let $f^\star$ be a solution to the BGKBE \eqref{eq:BGKBE}.
Multiplying \(m\times\)\eqref{eq:BGKBE}  and integrating over $\Xi=\Rz^d$ yields
\begin{align} \label{eq:mass}
  \partial_t\rho_{f^\star} + \bm{\nabla}_{\bm{x}} \!\cdot\! \left( \rho_{f^\star} \bm{u}_{f^\star} \right) + \underbrace{\int\limits_{\Rz^d} \bm{F}\!\cdot\!\bm{\nabla}_{\bm{v}} f^\star\,\mathrm{d}\bm{v}}_{=0}   &= -\frac{1}{\tau}\underbrace{\left( \rho_{f^\star}-\rho_{M_{f^\star}^{\mathrm{eq}}}\right)}_{\overset{\eqref{eq: boltz statistic equilibrium 7}}{=}0} \\
\iff   
\partial_t\rho_{f^\star} + \bm{\nabla}_{\bm{x}} \!\cdot\! \left( \rho_{f^\star} \bm{u}_{f^\star} \right) &= 0  \quad \text{in }I\times\Omega,
\end{align}
where the force term nulls out (cf. \cite[Corollary 5.2]{krause2010fluid} with $g = 1$ and $a=F$ in the respective notation). 
Dividing by the constant $\rho_{f^\star}$, the conservation of mass in the NSE is verified. 
To balance momentum, we integrate \(m\bm{v} \times \)\eqref{eq:BGKBE} over $\Xi=\Rz^d$ and obtain that
\begin{align} \label{eq:momentum}
\partial_ t \left( \rho_{f^\star} \bm{u}_{f^\star} \right) + \bm{\nabla}_{\bm{x}} \!\cdot\! \mathbf{P}_{f^\star} + \left(\rho_{f^\star}\bm{u}_{f^\star} \!\cdot\! \bm{\nabla}_{\bm{x}}\right) \bm{u}_{f^\star} + \bm{F} &= 0 \quad \text{in } I \times \Omega .
\end{align}
The derivation of \eqref{eq:momentum} closely follows a standard procedure documented for example in \cite[Subsection 1.3.1]{krause2010fluid}. 
Finally, via \eqref{eq:momentum}\(/ \rho_{f^\star}\) a balance law of momentum in conservative form is recovered.  
Thus, when suitably defining and simplifying \(\mathbf{P}_{f^\star}\) according to the assumption of incompressible Newtonian flow, the incompressible NSE appears as the diffusive limiting system.

\subsubsection{Step 2: Incompressible limit}
The incompressible limit regime of the BGKBE \eqref{eq:BGKBE} is obtained via aligning parameters to the diffusion terms (see e.g.\ \cite{krause2010fluid}). 
Let $l_\mathrm{f}$ be the mean free path, $\overline{c}$ the mean absolute thermal velocity, and \(\nu>0\) a kinematic viscosity. 
Assuming that a characteristic length $L$ and a characteristic velocity $U$ are given, we define the Knudsen number, the Mach number and the Reynolds number, respectively 
\begin{align} 
\Kn & \coloneqq  \frac{ l_{\mathrm{f}} }{L} , \\
\Ma & \coloneqq \frac{U}{c_\mathrm{s}} , \\
\Rey & \coloneqq \frac{U L}{\nu}.
\end{align}
These non-dimensional numbers relate as 
\begin{align}
\Rey = \frac{l_{\mathrm{f}} c_\mathrm{s}}{\nu} \frac{\Ma}{\Kn} =\sqrt{\frac{24}{\pi}} \frac{\Ma}{\Kn} ,
\end{align}
via defining \(\nu \coloneqq \pi \overline{c} l_{\mathrm{f}} / 8\) and the isothermal speed of sound \(c_\mathrm{s} \coloneqq \sqrt{3 RT}\) (see also \cite{saint-raymond2003bgk} and references therein). 
\begin{definition}
    To link the mesoscopic distributions with the macroscopic continuum we inversely substitute \(c_{\mathrm{s}}\) with an artificial parameter \(\epsilon \in \Rzsp\) through 
\begin{align}
c_{\mathrm{s}} \mapsfrom \frac{1}{\epsilon} . 
\end{align}
Here, and in the following the symbol \(\mapsfrom\) denotes the assignment operator. 
\end{definition}
In the limit $\epsilon \searrow 0$, the incompressible continuum is reached, since $\Kn$ and $\Ma$ tend to zero while $\Rey$ remains constant \cite{saint-raymond2003bgk}. 
Based on that, we assign
\begin{align} \label{eq:c overline}
\overline{c}= \sqrt{\frac{8k_{\mathrm{B}} T}{m \pi}} \mapsfrom \sqrt{\frac{8}{3 \pi}} \frac{1}{\epsilon} .
\end{align}
Further, 
\begin{align} 
 l_{\mathrm{f}} \mapsfrom \sqrt{\frac{24}{\pi}} \nu \epsilon
\end{align}
and \eqref{eq:c overline} unfold the relaxation time  
\begin{align} 
  \tau = \frac{l_{\mathrm{f}}}{\overline{c}} \mapsfrom 3 \nu \epsilon^2 .
\end{align}
\begin{definition}
Consequently, we define the \(\epsilon\)-parametrized BGKBE \eqref{eq:BGKBE} as
\begin{align} \label{eq: lbm discrete velocity 3}
\frac{\mathrm{D}}{\mathrm{D} t} f = - \frac 1 {3\nu \epsilon^2} \left( f-M_{f}^{\mathrm{eq}} \right)  \quad \text{in } \mathfrak{R} , 
\end{align}
where the \(\epsilon\)-parametrized Maxwellian distribution evaluated at $(n_f,\bm{u}_f)$ is
\begin{align} \label{eq:MaxwellianParametrized}
M_{f}^{\mathrm{eq}} = \frac{n_f \epsilon^d}{\left( \frac 2 3 \pi \right)^{\frac{d}{2}}} \exp\left( - \frac{3}{2} \left( \bm{v} \epsilon - \bm{u}_f  \epsilon\right)^2\right) \quad \text{in } \mathfrak{R} .
\end{align}
\end{definition}
The BGKBE \eqref{eq: lbm discrete velocity 3} transforms to
\begin{align} \label{eq: boltz statistic limit 8}
 f & = M_{f}^{\mathrm{eq}} - 3\nu \epsilon^2 \matD f \quad \text{in } \mathfrak{R} .
\end{align}
Repeating \((\matDil)\)\eqref{eq: boltz statistic limit 8} gives
\begin{align} \label{eq: boltz statistic limit 9}
 \matD f = \matD M_{f}^{\mathrm{eq}} - 3 \nu \epsilon^2 \left( \matD \right)^{2} f  \quad \text{in } \mathfrak{R} .
\end{align}
The expression \eqref{eq: boltz statistic limit 9} substitutes $(\matDil) f$ in \eqref{eq: boltz statistic limit 8}. 
Thus
\begin{align} 
 f  =  M_{f}^{\mathrm{eq}} - 3 \nu \epsilon^2 \matD M_{f}^{\mathrm{eq}} + \left( 3\nu \epsilon^2 \matD \right)^{2} f   \quad \text{in } \mathfrak{R}.
\end{align}
Subsequent repetition produces higher order terms and substitutions. 
The appearing power series in $\epsilon$ around $t$ reads
\begin{align} \label{eq: boltz statistic limit 11}
 f  = \sum\limits_{i=0}^\infty  \left( -3\nu \epsilon^2 \matD \right)^i M_{f}^{\mathrm{eq}} \quad \text{in } \mathfrak{R} .
\end{align}
\begin{remark}
Up to lower order, equation \eqref{eq: boltz statistic limit 11} can also be obtained via Maxwell iteration \cite{zhao2017maxwell} and references therein. 
The derivation in \cite{zhao2017maxwell} is however based on an initial Taylor expansion of the material derivative, whereas the present formulation starts with repeated application of the material derivative. 
Further, similarities to classical Chapman--Enskog expansion (see e.g.\ \cite{wolf2000lattice} and references therein) are present. 
Comparisons of several expansion techniques for a discretized model based on the BGKBE can be found for example in \cite{caiazzo2009comparison}. 
\end{remark}

\subsubsection{Step 3: Newton's hypothesis}
For a solution $f^\star$ of \eqref{eq:BGKBE} the remaining stress tensor $\mathbf{P}_{f^\star}$ in \eqref{eq:momentum} has to fulfill Newton's hypothesis   
\begin{align}\label{eq:newtonsHyp}
\mathbf{P}_{f^\star} \overset{!}{=} - p_{f^\star} \mathbf{I}_{d} + 2 \nu \rho \mathbf{D}_{f^{\star}} + \mathcal{O}\left( \epsilon^{b}\right) \quad \text{in } I\times\Omega
\end{align}
up to order \(b>0\), where the rate of strain is denoted by
\begin{align}
\mathbf{D}_{f} = \frac{1}{2} \left[  \bm{\nabla}_{\bm{x}} \bm{u}_{f} +  \left( \bm{\nabla}_{\bm{x}} \bm{u}_{f} \right)^{\mathrm{T}} \right] .
\end{align}
\begin{proposition}
    With a cutoff at order \(b=2\) we formally obtain
\begin{align}
\mathbf{P}  = p \mathbf{I}_{d} - 2 \nu \rho\mathbf{D} .
\end{align}
\end{proposition}
\begin{proof}
    Equation \eqref{eq: boltz statistic limit 11}, provides an ansatz
\begin{align} \label{eq: boltz statistic limit 12}
 f ^\star = M_{f^\star}^{\mathrm{eq}} - 3\nu \epsilon^2 \matD M_{f^\star}^{\mathrm{eq}} \quad \text{in } \mathfrak{R} .
\end{align}
Due to the assumption that for $\epsilon \searrow 0$ higher order terms become sufficiently small, the order \(b\) is large enough. 
To validate \eqref{eq:newtonsHyp}, the stress tensor is computed by its definition \eqref{eq: boltz statistic f stress}. 
Unless stated otherwise, \(f\)-indices at moments are omitted below. 
First, the material derivative and \eqref{eq:mass} is used to obtain 
\begin{align}\label{eq: boltz statistic limit 13}
\matD M_{f}^{\mathrm{eq}} 
&= 
\left( 
 			\frac{1}{\rho} \matD \rho 
 			+ 3 \epsilon^2 \bm{c} \!\cdot\! \matD \bm{u} 
			- \frac{3 \epsilon^2}{m} \bm{c} \!\cdot\! \bm{F}  			
 			\right) M_{f}^{\mathrm{eq}}  \nonumber \\
&= 
\left[
			\frac{1}{\rho} \left( \partial_{t} + \bm{v}  \!\cdot\!  \bm{\nabla}_{\bm{x}} \right) \rho 
			+ 3 \epsilon^2 \bm{c}  \!\cdot\!   \left( \partial_{t} + \bm{v}  \!\cdot\!  \bm{\nabla}_{\bm{x}} \right) \bm{u} 
			- \frac{3 \epsilon^2 }{m} \bm{c}  \!\cdot\!  \bm{F}		
			\right] M_{f}^{\mathrm{eq}} \nonumber \\
&= 
\left[
			\frac{1}{\rho} \left( -\bm{u} \!\cdot\! \bm{\nabla}_{\bm{x}}\rho - \rho \bm{\nabla}_{\bm{x}} \!\cdot\! \bm{u}  + \bm{v}  \!\cdot\!  \bm{\nabla}_{\bm{x}} \rho \right) 
+ 3 \epsilon^2 \bm{c}  \!\cdot\!   \left( \partial_{t} + \bm{v}  \!\cdot\!  \bm{\nabla}_{\bm{x}} \right) \bm{u} 
			- \frac{3 \epsilon^2}{m} \bm{c}  \!\cdot\!  \bm{F} 			
  			\right] M_{f}^{\mathrm{eq}} \nonumber \\
&= \biggl[ 
			-\underbrace{\bm{\nabla}_{\bm{x}} \!\cdot\! \bm{u}}_{=:~a_f} 
			+ \smash{\underbrace{\frac{\bm{c}}{\rho}  \!\cdot\!  \bm{\nabla}_{\bm{x}} \rho}_{=:~b_f} 
			+ \underbrace{3 \epsilon^2 \bm{c} \!\cdot\! \partial_{t} \bm{u}}_{=:~c_f} 
			+ \underbrace{3 \epsilon^2 \bm{c} \!\cdot\! \left( \bm{v} \!\cdot\! \bm{\nabla}_{\bm{x}}\right) \bm{u}}_{=:~d_f}} 
			- \underbrace{\frac{3 \epsilon^2 \bm{c} }{m}  \!\cdot\!  \bm{F} }_{=:~e_{f}}		
			\biggr] M_{f}^{\mathrm{eq}} 
\end{align}
in $\mathfrak{R}$, where 
\begin{align}\label{eq:relVelo}
\bm{c} & \coloneqq \bm{v} - \bm{u}
\end{align}
defines the relative velocity. 
Plugging the derivative \eqref{eq: boltz statistic limit 13} into \eqref{eq: boltz statistic limit 12} gives
\begin{align} \label{eq: boltz statistic limit 14}
 f &= M_{f}^{\mathrm{eq}}\left[1 - 3 \epsilon^{2} \nu  \left(-a_f+b_f+c_f+d_f+e_f\right)\right] \quad \text{in } \mathfrak{R} .
\end{align}
Second, the velocity integrals of terms \(a_{f}, b_{f}, \ldots , e_{f}\) are individually evaluated. 
We use the symmetry of $M_{f}^{\mathrm{eq}}$ and that $M_{f}^{\mathrm{eq}}/n$ is a normal distribution with covariance $1/(3 \epsilon^2) \mathbf{I}_d$. 
In $I\times\Omega$ and for any $i,j,k,l\in \left\{1,2,...,d\right\}$ holds
\begin{align} 
m \int\limits_{\Rz^d} c_i c_j M_{f}^{\mathrm{eq}}\,\mathrm{d} \bm{v} 
 &= \frac{\rho}{3 \epsilon^2} \delta_{ij} , \label{eq:cc} \\
m \int\limits_{\Rz^d} c_i c_j c_k M_{f}^{\mathrm{eq}}\,\mathrm{d}\bm{v} 
 &= 0 , \label{eq:ccc} \\
m\int\limits_{\Rz^d} c_i c_j c_{k} v_l M_{f}^{\mathrm{eq}}\,\mathrm{d}\bm{v} 
 &= \frac{\rho}{9 \epsilon^4} \left( \delta_{ij}\delta_{kl} + \delta_{ik}\delta_{jl} + \delta_{il}\delta_{jk} \right)  .\label{eq:cccv}												
\end{align}
Hence, we obtain 
\begin{align}										     
 m \int\limits_{\Rz^d} c_i c_j a_f M_{f}^{\mathrm{eq}}\,\mathrm{d}\bm{v} 
 &= \left( m \int\limits_{\Rz^d} c_i c_j M_{f}^{\mathrm{eq}}\,\mathrm{d}\bm{v} \right) \partial_{x_k} u_k \stackrel{\eqref{eq:cc}}{=} \frac{\rho}{3 \epsilon^2} \partial_{x_k} \bm{u}_k  , \\
 m \int\limits_{\Rz^d} c_i c_j b_f M_{f}^{\mathrm{eq}}\,\mathrm{d}\bm{v} &= \left( m \int\limits_{\Rz^d} c_i c_j c_k M_{f}^{\mathrm{eq}}\,\mathrm{d}\bm{v} \right) \frac{1}{\rho}\partial_{x_k} \rho \stackrel{\eqref{eq:ccc}}{=} 0 , \\
 m \int\limits_{\Rz^d} c_i c_j c_f M_{f}^{\mathrm{eq}}\,\mathrm{d}\bm{v} &= \left( m \int\limits_{\Rz^d} c_i c_j c_k M_{f}^{\mathrm{eq}}\,\mathrm{d}\bm{v} \right) 3\epsilon^2 \partial_{t} u_k  \stackrel{\eqref{eq:ccc}}{=} 0 , \\									 
 m\int\limits_{\Rz^d} c_i c_j d_f M_{f}^{\mathrm{eq}}\,\mathrm{d}\bm{v} &= \left( m \int\limits_{\Rz^d} c_i c_j c_{k} v_l M_{f}^{\mathrm{eq}}\,\mathrm{d}\bm{v} \right) 3 \epsilon^2  \partial_{x_l} u_k  \nonumber \\
                         &\stackrel{\eqref{eq:cccv}}{=} 												\frac{\rho}{3 \epsilon^2} \left( \delta_{ij}\delta_{kl} + \delta_{ik}\delta_{jl} + \delta_{il}\delta_{jk} \right) \partial_{x_l} u_k  , \\					
 m\int\limits_{\Rz^d} c_i c_j e_f M_{f}^{\mathrm{eq}}\,\mathrm{d}\bm{v} & = \left( m \int\limits_{\mathbb{R}^{d}} c_i c_j c_k M_{f}^{\mathrm{eq}} \,\mathrm{d} \bm{v} \right) \frac{3\epsilon^2}{m} F_k \stackrel{\eqref{eq:ccc}}{=} 0 .
\end{align}
Third, for any $i,j \in \left\{1, 2, ..., d \right\}$, the component $P_{ij} \coloneqq ( \mathbf{P} )_{i,j}$ can be computed in \(\mathfrak{R} \). 
Reordering its terms, we achieve 
\begin{align} 
P_{ij} 
&=  m \int\limits_{\Rz^d} c_i c_j  \left[ 1- 3 \nu \epsilon^{2} \left(-a_f+b_f+c_f+d_f + e_{f} \right) \right] M_{f}^{\mathrm{eq}}\,\mathrm{d}\bm{v}   \nonumber\\
&= p \delta_{ij} 
 			-  3 \nu \epsilon^{2} \left[
			- \frac{\rho}{3 \epsilon^2} \partial_{x_k} u_k 
			+ \partial_{x_l} u_k \frac{\rho}{3 \epsilon^2} \left( \delta_{ij}\delta_{kl} + \delta_{ik}\delta_{jl} + \delta_{il}\delta_{jk} \right) 
		 \right]      \nonumber \\
        &= p \delta_{ij} + \nu \rho \left[ \delta_{ij} \partial_{x_k} u_k - \partial_{x_l} u_k \left( \delta_{ij}\delta_{kl} + \delta_{ik}\delta_{jl} + \delta_{il}\delta_{jk}\right) \right] \nonumber \\
        &= p \delta_{ij} -  \nu \rho \left( \partial_{x_{i}} u_j + \partial_{x_j} u_i \right),
\end{align}
which proves the claim. 
\end{proof}

\begin{remark}
Extending the formal result, the vanishing of higher order terms in the hydrodynamic limit \(\epsilon \searrow 0\) is rigorously proven in \cite[Notation: \(\varepsilon\) instead of \(\epsilon\)]{saint-raymond2003bgk} for the case $\Omega=\Rz^3$ and initial condition
\begin{align}
f_{\epsilon}(0,\bm{x},\bm{v}) = M \left( 1 + \epsilon g^{0}_{\epsilon}\left(\bm{x}, \bm{v}\right)\right)
\end{align}
close to an absolute equilibrium 
\begin{align}
M (\bm{v}) = \frac{1}{(2\pi)^{d/2}} \exp\left( -\frac{|\bm{v}|^2}{2}\right) 
\end{align}
with \(\bm{u}_{f} = \bm{0}\),  and \(\rho_{f} = 1 = RT\) and initial fluctuations \(g_{\epsilon}^{0}\). 
There, solutions \(f_{\epsilon}\) to the \(\epsilon\)-scaled BGKBE are passed to the limit, where the corresponding velocity moments \(\bm{u}_{f_{\epsilon}}\) are consequently identified as limiting to Leray's weak solutions \cite{leray1934sur} of the incompressible NSE \cite[Theorem 1.2]{saint-raymond2003bgk}. 
It is to be noted that we adapted the initial wording "hydrodynamic limit" from \cite{saint-raymond2003bgk} toward the more generic term "diffusive limit" to underline the presence of diffusion terms in the limiting equation. 
\end{remark}
For the sake of clarity, we recall the main result of Saint-Raymond \cite{saint-raymond2003bgk} in the present notation. 
It is to be stressed that the present work neglects the additional temperature equation appearing in the limit via imposing an ideal gas.  
Let the Hilbert space \(L^{2}(\mathbb{R}^{d}, M \,\mathrm{d} \bm{v})\) be defined by the scalar product 
\begin{align}
\left( f, g\right) \mapsto \int\limits_{\mathbb{R}^{d}} f\left( \bm{v} \right) g\left(\bm{v} \right) M \left( \bm{v} \right) \,\mathrm{d} \bm{v}, 
\end{align}
where \(M\,\mathrm{d}\bm{v}\) is a positive unit measure on \(\mathbb{R}^{d}\) which allows the definition of an average 
\begin{align}
\langle \xi \rangle \coloneqq \int\limits_{\mathbb{R}^{d}} \xi\left(\bm{v}\right) M\left( \bm{v} \right) \,\mathrm{d} \bm{v}
\end{align}
over \(\Xi = \mathbb{R}\) for any integrable function \(\xi\).  
\begin{definition}
For any pair \((f,g)\) of functions which are measurable and almost everywhere non-negative on \(\mathfrak{P}\), define the relative entropy
\begin{align}
H \left( f \mid g\right) \coloneqq \int\limits_{\mathbb{R}^{d}}\int\limits_{\mathbb{R}^{d}} \left( f \log \left( \frac{f}{g}\right) - f + g \right) \,\mathrm{d}\bm{v}\,\mathrm{d}\bm{x} \geq 0 
\end{align}
\end{definition}
Based on that, we recall the following preparative result from \cite{saint-raymond2002discrete}. 
Note that locally integrable functions on \(X \subseteq \mathbb{R}^{d}\) and Sobolev spaces based on \(L^{2}\) are denoted with 
\begin{align}
L^{1}_{\mathrm{loc}}\left(X\right) &= \left\{ f\colon X \to \mathbb{R} \; \left\vert \;  \forall \bm{x} \in X \, \exists r>0 \colon \, B_{r}(\bm{x}) \subseteq X  \, \wedge \, f\vert_{B_{r}(\bm{x})} \in L^{1}(B_{r}(\bm{x})) \right.\right\} , \\
H^{k}(X) = W^{k,2}(X) & = \left\{ f \in L^{2}(X)  \; \left\vert \; \forall \vert \alpha\vert \leq k \, \exists \text{ weak derivative }\partial^{\alpha} f \in L^{2}(X)  \right. \right\},  
\end{align}
respectively. 
The latter are Hilbert spaces and the dual of \(H^{k}(X)\) is denoted with \(H^{-k}(X)\). 
\begin{theorem}
Let \(\epsilon>0\) and \(0 \leq f_{\epsilon}^{0} \in L_{\mathrm{loc}}^{1} ( \mathfrak{P}) \) such that the entropy is bounded \(H (f_{\epsilon}^{0} \mid M ) < \infty\). 
Then \(\exists\) \(f_{\epsilon}\) a global, non-negative, and weak solution of \eqref{eq: lbm discrete velocity 3} which fulfills 
\begin{align}
f_{\epsilon} - M \in C\left( \mathbb{R}_{>0}, L^{2} \left( \mathfrak{P} \right) + L^{1}\left( \mathfrak{P} \right) \right) 
\end{align}
and \(\forall t > 0\):
\begin{align}
H \left( f_{\epsilon} (t) \mid M \right) + \frac{1}{\epsilon^{2}\nu}\int\limits_{0}^{t}\int\limits_{\mathbb{R}^{d}}\int\limits_{\mathbb{R}^{d}} D\left( f_{\epsilon} \right) \left(s\right) \,\mathrm{d}\bm{v}\,\mathrm{d}\bm{x}\,\mathrm{d}s \leq H\left( f_{\epsilon}^{0} \mid M \right) ,
\end{align}
where the dissipation \(D ( f_{\epsilon} ) \) is defined by 
\begin{align}
D\left( f_{\epsilon}\right) = \left( M_{f_{\epsilon}}  - f_{\epsilon} \right) \log \left( \frac{Mf_{\epsilon}}{f_{\epsilon}}\right) \geq 0 .
\end{align}
Additionally, the weak solution fulfills the typical moment integrated conservation laws \eqref{eq: boltz statistic equilibrium 7}, \eqref{eq: boltz statistic equilibrium 8}, and \eqref{eq: boltz statistic equilibrium 10} at zeroth, first, and second order, respectively. 
\end{theorem}
\begin{proof}
This theorem has been stated and proven in \cite{saint-raymond2002discrete}.
\end{proof}
Via a boundedness assumption in \(L^{2}( \mathfrak{P}, \,\mathrm{d}\bm{x}M\,\mathrm{d}\bm{v} ) \) of initial sequential fluctuation data \((g_{\epsilon}^{0})\) defined from 
\begin{align}
g_{\epsilon}^{0} = \frac{1}{\epsilon} \left( \frac{f_{\epsilon}^{0}}{M}- 1 \right) ,
\end{align}
a constantly prefactored entropy bound \(C_{0} \epsilon^{2}\) is obtained and in turn weak compactness on \((Mg_{\epsilon})\) holds in \(L^{1}_{\mathrm{loc}} ( \mathbb{R}_{>0}\times \mathbb{R}^{d}, L^{1}(\mathbb{R}^{d}))\) \cite{bardos1993fluid,saint-raymond2003bgk}. 
Within this setting, for \(1\geq p \geq \infty \) define 
\begin{align}
w-L^{p} \coloneqq \begin{cases} 
\text{weak topology } \sigma \left( L^{p}, L^{p^{\prime}} \right) \quad & \text{if } p<\infty , \\
\text{weak-\(^\ast\) topology } \sigma \left( L^{\infty}, L^{1} \right)  \quad & \text{if } p=\infty ,
\end{cases}
\end{align}
to restate the following main result of \cite{saint-raymond2003bgk}.
\begin{theorem}
Let \((g_{\epsilon}^{0})\) be a family of measurable functions on \(\mathfrak{P}\) which satisfy 
\begin{align}
1 + \epsilon g_ {\epsilon}^{0} & \geq 0 \text{ almost everywhere}, 	\\
H\left( M\left( 1 + \epsilon g_{\epsilon}^{0} \right) \mid M \right) & \leq C_{0} \epsilon^{2}, 
\end{align}
and with \(\bm{\nabla}_{\bm{x}} \cdot \bm{u}_{0} = 0 \) additionally fulfill that 
\begin{align}
\left\langle g_{\epsilon}^{0} \bm{v} \right\rangle & \xrightharpoonup[w - L^{2}\left( \mathbb{R}^{d}\right)]{\epsilon\searrow 0} \bm{u}_{0} 
\end{align}
Further, let \(f_{\epsilon} = M( 1+ \epsilon g_{\epsilon} )\) be a solution of \eqref{eq: lbm discrete velocity 3}. 
Then 
\begin{align}
\exists \rho , \bm{u} \in L^{\infty} \left( \mathbb{R}_{>0}, L^{2}\left( \mathbb{R}^{d} \right) \right) \cap L^{2}\left( \mathbb{R}_{>0}, H^{1}\left( \mathbb{R}^{d} \right) \right)
\end{align}
we have weak convergence such that, modulo a zero-limiting subsequence \((\epsilon_{n})\), 
\begin{align}
\left\langle g_{\epsilon} \right\rangle & \xrightharpoonup[w - L_{\mathrm{loc}}^{1}\left( \mathbb{R}_{>0} \times \mathbb{R}^{d} \right)]{\epsilon\searrow 0} \rho, \\
\left\langle g_{\epsilon} \bm{v} \right\rangle & \xrightharpoonup[w - L_{\mathrm{loc}}^{1}\left( \mathbb{R}_{>0} \times \mathbb{R}^{d} \right)]{\epsilon\searrow 0} \bm{u}	.
\end{align}
In addition, \( \rho\) and \(\bm{u}\) are weak solutions of the incompressible NSE \eqref{eq:incNSE}, where the pressure is determined from the solenoidal condition. 
\end{theorem}
\begin{proof}
An extended version of this theorem (considering nonuniform temperature) has been stated and proven in \cite{saint-raymond2003bgk}.
\end{proof}

Anticipating the discretization of velocity, space, and time, we explicitly state the \(\epsilon\)-parametrized continuous BGKBE in terms of a family of differential equations. 
\begin{definition}
With  \(\epsilon^{2}\times\)\eqref{eq: lbm discrete velocity 3}, we construct the family of BGKBEs
\begin{align} \label{eq: lbm discrete velocity 4}
 \mathcal{F}:= \left( \epsilon^2 \matD f^{\epsilon} + \frac 1 {3\nu} \left( f^{\epsilon}-M_{f^{\epsilon}}^{\mathrm{eq}} \right)=0 \quad\text{in } \mathfrak{R} \right)_{\epsilon>0} , 
\end{align}
where the conservable \(f^{\epsilon}\) in \eqref{eq: lbm discrete velocity 4} now displays an upper index \(\epsilon\) to stress the dependence on the artificial parametrization. 
\end{definition}

\section{Numerical methodology}\label{sec:limitConsistency}

\subsection{Limit consistency}

In general, the term consistency is often used to describe a limiting approach for discrete equations to a single continuous one. 
Since, in the present framework we instead aim to verify conformity of families of equations in terms of a weak solution limits under discretization (see Figure~\ref{fig:limitConsAbs}), we adjust the definition of consistency below. 
The necessity of doing so has been initially motivated in \cite[Definition 2.1.]{krause2010fluid}. 
Further supporting this redefinition, we stress the modular character of the notion of limit consistency. 
Based on the here proposed methodology, we enable the consistency analysis of discretized kinetic equations in terms of limiting towards a target PDE with a previously imposed parametrization or scaling. 
In the present context, we use the novel methodology to obtain a limit consistent, intrinsically parallel discrete evolution equation that forms the centerpiece of many LBMs.

\begin{figure}[ht!]
\centering
\includegraphics[scale=1]{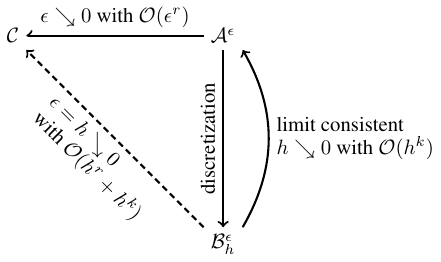}
\caption{Schematic concept of limit consistency.  
The targeted PDE of is denoted with \(\mathcal{C}\), the kinetic equation or relaxation system is denoted with \(\mathcal{A}^{\epsilon}\), the kinetic scheme or relaxation scheme is denoted with \(\mathcal{B}^{\epsilon}_{h}\). 
The diagonal limit \(\mathcal{B}_{h}^{\epsilon} \dashrightarrow \mathcal{C}\) for \(\epsilon = h \searrow 0\) is presently focused.}
\label{fig:limitConsAbs}
\end{figure}

\begin{remark}
Compared to previous approaches which regard the space-time and velocity discretizations by separate methods, the distinct feature of the present methodology is thus manifested in its generic modularity. 
Hence, we enable to use thermodynamic information only if necessary to approximate the model PDE. 
Notably, in \cite{simonis2023pde} limit consistency is successfully used for analyzing both, convergence of discretized, mathematically abstract relaxation systems without thermodynamic information as well as limit information of discretized, BGKBE-based models. 
\end{remark}

\begin{figure}[ht!]
\centering
\includegraphics[scale=1]{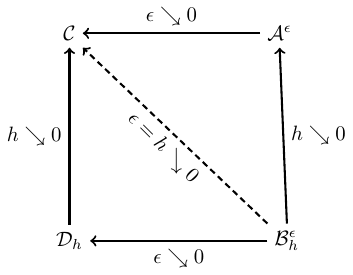}
\caption{Schematic concept of coupled kinetic (relaxation) and discretization limits.
The targeted PDE of is denoted with \(\mathcal{C}\), the kinetic equation or relaxation system is denoted with \(\mathcal{A}^{\epsilon}\), the kinetic scheme or relaxation scheme is denoted with \(\mathcal{B}^{\epsilon}_{h}\) and the corresponding macroscopic scheme or relaxed scheme is \(\mathcal{D}_{h}\).
Here, \(h\) defines a velocity or space-time discretization.
The diagonal limit \(\mathcal{B}_{h}^{\epsilon} \dashrightarrow \mathcal{C}\) for \(\epsilon = h \searrow 0\) is presently focused.}
\label{fig:asymptoticPreserving}
\end{figure}
    
\begin{remark}
In order to underline the novelty of the present approach, we compare the idea to established ones with the help of Figure~\ref{fig:asymptoticPreserving}. 
The illustration is based on a given relaxation or kinetic limit \(\mathcal{A}^{\epsilon} \to \mathcal{C}\) for \( \epsilon \searrow 0 \) from a relaxation system or kinetic equation \(\mathcal{A}^{\epsilon}\) to its targeted PDE \(\mathcal{C}\). 
For example, the well-established property of asymptotic preserving defines whether a stable and consistent space-time discretization \(\mathcal{D}_{h}\) exists in the macroscopic limit \(\epsilon \searrow 0 \) of a space-time discretized relaxation scheme \(\mathcal{B}^{\epsilon}_{h}\) \cite{guo2023unified,junk2001finite,jin2012asymptotic} with discretization parameter \(h>0\). 
If this is the case, the formal limit equality of equations
\begin{align}\label{eq:formalLimitEquAP}
\lim\limits_{h\searrow 0} \underbrace{ \left( \lim\limits_{\epsilon \searrow 0} \mathcal{B}_{h}^{\epsilon} \right)}_{=\mathcal{D}_{h}} 
\stackrel{!}{=} \lim\limits_{\epsilon\searrow 0} \underbrace{\left( \lim\limits_{h\searrow 0} \mathcal{B}_{h}^{\epsilon} \right) }_{=\mathcal{A}^{\epsilon}} = \mathcal{C} 
\end{align}
should hold. 
In contrast to that, here the parameter \(\epsilon\) is glued to the grid, i.e.\ \(\epsilon \mapsfrom h\), such that, in the context of LBMs we work with the formal limit equality 
\begin{align}\label{eq:formalLimitEquLC}
\lim\limits_{\substack{\epsilon = h \\ h \searrow 0}} \mathcal{B}_{h}^{\epsilon}  
\stackrel{!}{=} \lim\limits_{\epsilon\searrow 0} \underbrace{\left( \lim\limits_{h\searrow 0} \mathcal{B}_{h}^{\epsilon} \right) }_{=\mathcal{A}^{\epsilon}} = \mathcal{C} , 
\end{align}
rather than with \eqref{eq:formalLimitEquAP}. 
Drawing the analogy to limits and continuity of multivariate functions \cite{folland2001advanced} (via \(\mathcal{B}^{\epsilon}_{h} = \mathcal{B} (\epsilon, h)\)), we expect that other shapes of paths than \(\epsilon = h\) can be used. 
In fact, the mapping function \(h \mapsfrom \mathcal{O}(\epsilon^{\alpha_{0}})\) is analyzed in \cite{guo2023unified}. 
The overall order \(\alpha_{0}\) is the minimum of exponents from space \(\triangle x\) and time \(\triangle t\) discretization in the order of \(\epsilon\), respectively, and leads to distinct features of the scheme \(\mathcal{B}^{\epsilon}_{h}\). 
Presently, we focus on the order at which the limit point is approximated by the diagonal path (\(\epsilon = h\)). 
Notably, hybrid schemes have been derived e.g.\ by Klar~\cite{klar1999relaxation}, which are based on a discrete velocity Boltzmann equation but completed with an asymptotic preserving discretization to achieve uniform functionality for all ranges in \(\epsilon\). 
\end{remark}

\begin{definition} \label{def:familyPDEs}
Let \(d \in \mathbb{N}\) and $X \subseteq \mathbb{R}^d$ with a discretization $X_h \subseteq X$ for any $h\in\mathbb{R}_{>0}$. 
Let $U(X)$ and $W_h(X_h)$ denote Hilbert spaces on \(X\) and \(X_{h}\), respectively, where $W_h$ contains the grid functions of \(\{v_{h}\colon X_{h} \to \mathbb{R}\} \). 
Via  
\begin{align}
\mathcal{A}^{\epsilon} & = \left( A^{\epsilon} \left( \cdot \right)=0 \quad \text{in } U\right)_{\epsilon>0}, \label{eq:familyAeps} \\
\mathcal{B}_{h}^{\epsilon} &= \left( B_{h}^{\epsilon} \left( \cdot \right)=0 \quad \text{in } W_h\right)_{\epsilon>0, h>0},  \label{eq:familyBepsh}
\end{align}
families of PDEs are defined by continuous and discrete operators \(A^{\epsilon}\) and \(B_{h}^{\epsilon}\), respectively. 
Solutions to instances of \eqref{eq:familyAeps} and \eqref{eq:familyBepsh} are denoted with \(a^{\epsilon} \in U\) for all \(\epsilon >0 \) and \(b_{h}^{\epsilon} \in W_{h}\) for all \(\epsilon,h>0\), respectively. 
\end{definition}
Conforming to Definition~\ref{def:familyPDEs}, we occasionally adopt the notation 
\begin{align}
 \mathcal{C} = \left( C \left( \cdot \right) = 0 \quad \text{in } \tilde{U} \right)  \label{eq:familyC} 
\end{align}
for a single PDE with a solution \(c\) in the Hilbert space $\tilde{U}(X)$. 
\begin{definition}\label{def:PDEandSOLconv}
Let \(\mathcal{A}^{\epsilon}\) and \(\mathcal{C}\) be given as in \eqref{eq:familyAeps} and \eqref{eq:familyC}, respectively. 
The abstracted solution limit of a solution \(a^{\epsilon}\) to \(\mathcal{A}^{\epsilon}\) toward a solution \(c\) to \(\mathcal{C}\) for \(\epsilon \searrow 0\) is denoted with 
\begin{align}
a^{\epsilon} \rightharpoonup c
\end{align}
and defines convergence in the broadest sense (e.g.\ formal, weak or strong). 
The formal order \(r\geq 0\) of this convergence is \(A^{\epsilon} ( c ) \in \mathcal{O}(\epsilon^{r}) \). 
The information of both, the formal PDE convergence and the solution convergence is compressed in the notation 
\begin{align}
\mathcal{A}^{\epsilon} \xrightharpoonup[\mathcal{O}(\epsilon^{r})]{\epsilon\searrow 0} \mathcal{C}  . 
\end{align} 
\end{definition}
Based on the abstracted but formally determined background limit in Definition~\ref{def:PDEandSOLconv}, we propose the following specialized notion of consistency. 
\begin{definition}\label{def:limitConsistency}
Let \(\mathcal{A}^{\epsilon}\) admit an abstracted solution limit of order \(\mathcal{O}(\epsilon^{r})\) in \(\epsilon \searrow 0\) to a solution \(c\) of a PDE system \(\mathcal{C}\) as in Definition~\ref{def:PDEandSOLconv}. 
Then, $\mathcal{B}^{\epsilon}_{h}$ is called limit consistent of order $k>0$ to $\mathcal{A}^{\epsilon}$ in $W_{h}(X_{h})$, if for any fixed \(\epsilon>0\) holds that
\begin{enumerate}[label=(\roman*)]
\item \(B_{h}^{\epsilon} ( a^{\epsilon} |_{X_{h}} ) \in \mathcal{O} \left(h^k \right)\) in \(W_h\) ,  and 
\item \(k\geq r\) .
\end{enumerate}
The residual expression \(B_{h}^{\epsilon} ( a^{\epsilon} |_{X_{h}} )\) is called truncation error.  
\end{definition}
\begin{lemma}\label{pro:limSup}
Let $\mathcal{B}_{h}^{\epsilon}$ be limit consistent of order $k$ to $\mathcal{A}^{\epsilon}$ in $W_{h} (X_{h})$. 
Then for any fixed \(\epsilon >0\) we have 
\begin{align}
\left[ B_{h}^{\epsilon} \left( a^{\epsilon} \mid_{X_{h}} \right) \in \mathcal{O} \left(h^k \right) \quad\text{in } W_h \right] 
\iff 
\lim\limits_{h\searrow 0} \sup\limits_{x \in X_h} \left\vert \frac {B_{h}^{\epsilon} \left(a^{\epsilon} |_{X_{h}}\right) ( x ) }{ h^{k} } \right\vert  < \infty .
\end{align}
\end{lemma}
\begin{proof}
We interpret the operation 
\begin{align}
\cdot\vert_{X_{h}} : U \to W_{h}, f \mapsto f\vert_{X_{h}} 
\end{align}
as an interpolation which is exact at the grid nodes of \(X_{h}\). 
Let \(\epsilon >0\) be fixed. 
Forming the local truncation error of \(B_{h}^{\epsilon}\) with respect to \(A^{\epsilon}\), via insertion of the exact solution \(a^{\epsilon}\) evaluated at the grid nodes \cite{leveque2007finite}, gives 
\begin{align}
B_{h}^{\epsilon} \left( a^{\epsilon}\vert_{X_{h}} \right) = K h^{k} + \mathcal{O}\left( h^{k+1}\right), 
\end{align}
with a constant \(K<\infty\). 
Due to consistency, i.e.\ the local truncation nulling out for \(h\searrow 0\), we can limit 
\begin{align}
\lim\limits_{h\searrow 0} \left\| B_{h}^{\epsilon} \left( a^{\epsilon}\vert_{X_{h}} \right) \right\| = 0 ,
\end{align}
where \(\left\| \cdot \right\| \coloneqq \sup\limits_{x \in X_{h}} \left\vert \cdot \right\vert \) defines a supremum norm on \(W_{h}\).
Similarly, we have that 
\begin{align}
\lim\limits_{h\searrow 0} \left\| \frac{ B_{h}^{\epsilon} \left( a^{\epsilon}\vert_{X_{h}} \right)}{h^{k}}  \right\| = \lim\limits_{h\searrow 0} \left\| \frac{ K h^{k} + \mathcal{O}\left( h^{k+1}\right) }{h^{k}}  \right\| = K + \mathcal{O} ( 1 ) < \infty .  
\end{align}
\end{proof}
\begin{remark}
It is to be stressed, that the difference to classical consistency is with respect to the exact solution \(a^{\epsilon}\) being already parametrized in \(\epsilon\). 
Via the assignment of the artificial parameter \(\epsilon \mapsfrom h\) and the interpolation \(a^{\epsilon\mapsfrom h}\vert_{X_h}\) onto the grid nodes, the kinetic/relaxation parametrization is irreversibly coupled to the discretization. 
The process of discretization has thus to be consistent to or at least uphold this limit. 
If this is the case, the limit consistency implies classical consistency with concatenated orders.
\end{remark}

\begin{remark}
Note that, in Definition~\ref{def:limitConsistency}, we have purposely not specified the limit \(\mathcal{A}^{\epsilon} \xrightharpoonup[]{\epsilon\searrow 0} \mathcal{C} \) further. 
Dependent on the situation at hand, this limit can be e.g.\ weak or strong. 
For example, the former is the case when approximating weak solutions of the incompressible NSE~\eqref{eq:incNSE} with the BGKBE~\eqref{eq:BGKBE}~\cite{saint-raymond2003bgk} in diffusive scaling. 
The latter is given when using a relaxation system (or the corresponding BGK model \cite{simonis2020relaxation}), for the approximation of scalar, linear, \(d\)-dimensional ADE \cite{simonis2022constructing}. 
The limit can also be in terms of unique entropy solutions, if \(\bm{F}\) is nonlinear~\cite{bouchut2000diffusive}. 
\end{remark}

\begin{remark}
By Lemma \ref{pro:limSup}, we have identified $B_{h}^{\epsilon}(a^{\epsilon}|_{X_h})$ as the abstracted local truncation error 
\begin{align}
- T^{\epsilon}_{h} \coloneqq \underbrace{ B_{h}^{\epsilon} \left(b_{h}^{\epsilon}\right) }_{= 0 } - B_{h}^{\epsilon}\left(a^{\epsilon}\vert_{X_{h}}\right) 
\end{align}
(e.g.\ see \cite{leveque2007finite,bartels2015numerical}) with an additional relaxation limit running in the background. 
As a consequence, demanding stability seems natural to complete the convergence result. 
\end{remark}

Let the global error be defined by\phantomsection\label{sym:globalErrorRS} 
\begin{align}
E^{\epsilon}_{h} = b_{h}^{\epsilon} - a^{\epsilon}\vert_{X_{h}} . 
\end{align}
We cut off the Taylor expansion of \(B_{h}^{\epsilon} (b_{h}^{\epsilon})\) at \(a^{\epsilon}\vert_{X_{h}}\) given by 
\begin{align}
B_{h}^{\epsilon}(b_{h}^{\epsilon}) = \sum\limits_{n=0}^{\infty} \frac{1}{n!} \left( \partial_{b_{h}^{\epsilon}}\right)^{n} B_{h}^{\epsilon} \left( a^{\epsilon} \vert_{X_h} \right) \left(E_{h}^{\epsilon}\right) ^{n}
\end{align}
to obtain a linearized expression 
\begin{align}
J_{B_{h}^{\epsilon}} \left( a^{\epsilon} \vert_{X_{h}} \right) E_{h}^{\epsilon} = - T_{h}^{\epsilon} + \mathcal{O} \left( \| E_{h}^{\epsilon} \|^{2} \right)  ,
\end{align}
where \(J_{B_{h}^{\epsilon}}(a^{\epsilon}\vert_{X_{h}})\) denotes the Jacobian of the discrete operator \(B_{h}^{\epsilon}\) at exact solutions \(a^{\epsilon}\) of \(A^{\epsilon}\) evaluated on the grid. 
The nonlinear terms are gathered in \(\mathcal{O}(\| E_{h}^{\epsilon}\|^{2}) \). 
Following \cite{leveque2007finite} we can define the a notion of stability with respect to the linearized discretization.  

\begin{definition}\label{def:stabilityLeveque}
For fixed \(\epsilon\), the linearized discrete operator \(B_{h}^{\epsilon}\) is stable in some norm \(\|\cdot\|_{W_{h}}\) on \(W_{h}\) if its inverse Jacobian at the exact solution evaluated at \(X_{h}\) is uniformly bounded for \(h\searrow 0\) in the sense that there exist constants \(K>0\) and \(h_{0}\) such that 
\begin{align}
\left\|  \left( J_{B_{h}^{\epsilon}}\left( a^{\epsilon} \vert_{X_{h}} \right)\right)^{-1} \right\| \leq K  \quad \text{for all } h < h_{0} .
\end{align}
\end{definition}

\begin{remark}
In the context of the LBM, where \(B_{h}^{\epsilon}\) is the space-time-velocity discrete lattice Boltzmann equation (LBE) for \(\epsilon \mapsfrom h\), several previous works derived bounds for linearized amplification matrices in the sense of von Neumann (e.g.\ \cite{simonis2020relaxation}) and proved weighted \(L^2\)-stability \cite{junk2009weighted} for linearized collisions which admit a stability structure. 
As a matter of fact, nonlinear stability estimates for LBMs naturally involve the notion of entropy in a mathematical (relaxation) \cite{caetano2019result} or thermodynamical (dynamical system) \cite{boghosian2001entropic} sense. 
\end{remark}

\begin{remark}
Upon condition that the lattice Boltzmann discretizations are stable in some norm, limit consistency can be used to infer classical consistency and hence convergence \cite{leveque2007finite,tadmor2012review} toward the target PDE.  
Below, the overall notion of convergence is to be understood in terms of the kind of relaxation or kinetic background limit \(\mathcal{A}^{\epsilon} \xrightharpoonup[]{\epsilon\searrow 0} \mathcal{C} \) only (e.g.\ formal, weak or strong). 
\end{remark}

\begin{lemma}\label{lem:weakConvWithLC}
Let \(\mathcal{A}^{\epsilon}_{h}\) and \(\mathcal{B}_{h}^{\epsilon}\) be given as in Definition~\ref{def:limitConsistency} and let 
\begin{enumerate}[label=(\roman*)]
\item \(\mathcal{B}_{h}^{\epsilon}\) be limit consistent of order \(k\) to \(\mathcal{A}^{\epsilon}\), 
\item \(B_{h}^{\epsilon}\) be stable and linear in the terms of Definition~\ref{def:stabilityLeveque}. 
\end{enumerate}
Then we obtain an overall convergence result of solutions in the sense of 
\begin{align}
\mathcal{B}_{h}^{\epsilon} \xrightarrow[\mathcal{O}\left(\epsilon^{r}\right)|_{X_{h}} + \mathcal{O} \left( \epsilon^{k}\right)]{(\epsilon,h) \searrow (0,0)} \mathcal{C} 
\equiv 
\left( \mathcal{A}^{\epsilon} \xrightharpoonup[\mathcal{O}(\epsilon^{r})]{\epsilon \searrow 0} \mathcal{C} \right)
\circ 
\left(\mathcal{B}_{h}^{\epsilon} \xrightarrow[\mathcal{O}\left(h^{k}\right)]{h\searrow 0} \mathcal{A}^{\epsilon} \right) , 
\end{align}
where the symbol \(\equiv\) denotes arrow equality irrespective of the nature of the mappings. 
Further, if \(\epsilon = \iota(h)\) via \(\iota = \mathrm{id}\),  \(b^{\epsilon}_{h}\) converges at order \(r\) to \(c\).  
\end{lemma}
\begin{proof}
For fixed \(\epsilon > 0\), limit consistency of \(\mathcal{B}_{h}^{\epsilon}\) and the stability of \(B_{h}^{\epsilon}\) imply classically \cite{lax1956survey} that
\begin{align}
\left\| E_{h}^{\epsilon} \right\| 
& = 
\left\| - \left( J_{B_{h}^{\epsilon}} \left( a^{\epsilon}|_{X_{h}} \right) \right)^{-1} T_{h}^{\epsilon} + \mathcal{O}\left( \left\| E_{h}^{\epsilon} \right\|^{2} \right) \right\|\nonumber \\
& \leq 
\left\|  \left( J_{B_{h}^{\epsilon}} \left( a^{\epsilon}|_{X_{h}} \right) \right)^{-1} \right\| \left\| T_{h}^{\epsilon} \right\|\nonumber \\
& \leq K  \mathcal{O}\left( h^{k}\right) 
\end{align}
and hence
\begin{align}\label{eq:discreteSol}
b_{h}^{\epsilon} = a^{\epsilon}|_{X_{h}} + \mathcal{O}\left( h^{k} \right).
\end{align}
Similarly, from the background (relaxation or kinetic) limit, we have that 
\begin{align}\label{eq:relaxSol}
a^{\epsilon} = c + \mathcal{O}\left( \epsilon^{r} \right) . 
\end{align}
Combining \eqref{eq:discreteSol} and \eqref{eq:relaxSol} we obtain 
\begin{align} \label{eq:exactSol}
b_{h}^{\epsilon} = c|_{X_{h}} + \mathcal{O}\left( \epsilon^{r}\right)|_{X_{h}} +\mathcal{O}\left( h^{k} \right).
\end{align}
Now let \(\epsilon \mapsfrom \iota(h)\).\phantomsection\label{sym:assignEpsilon} 
Thus, since \(\mathcal{O}(\epsilon^{r})|_{X_{h}}\) are higher order terms interpolated on the grid nodes, we see that their leading order in \(h\) is \(r\). 
Conclusively, \eqref{eq:exactSol} becomes   
\begin{align}
b_{h}^{\epsilon} & = c|_{X_{h}} + \mathcal{O} \left( \iota (h)^{r} + h^{k} \right) \nonumber\\
& = c|_{X_{h}} + \mathcal{O}\left( h^{\mathrm{min}\left\{r,k\right\}} \right) \nonumber
\\
& = c|_{X_{h}} + \mathcal{O}\left( h^{r} \right) ,
\end{align}
due to the limit consistent discretization. 
\end{proof}

\begin{remark}
Having established the definition of limit consistency, we utilize it below to provide limit consistent discretizations of families of BGK Boltzmann equations, as suggested in Figure~\ref{fig:limitCons}. 
In particular, we elaborate the previous work of Krause \cite{krause2010fluid} via reconstructing the discretizations and validating the procedure in terms of limit consistency. 
\end{remark}

\begin{figure}[ht!]
\centering
\includegraphics[scale=1]{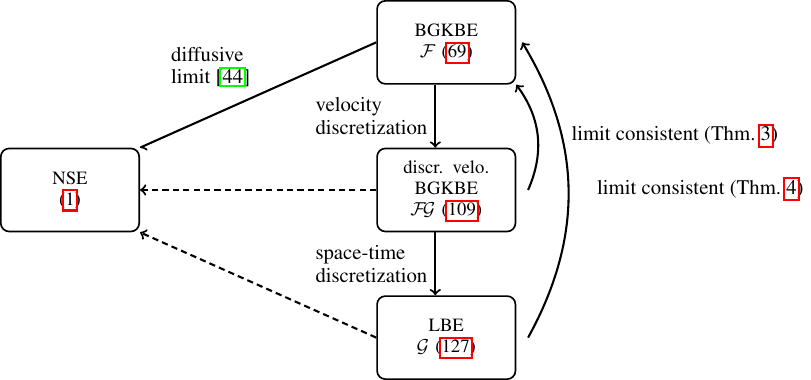}
\caption{Illustration of limit consistency for the present BGKBE discretizations.}
\label{fig:limitCons}
\end{figure}

\subsection{Discretization of velocity}

Classically, the discrete velocity BGKBE is a result of reducing the velocity space \(\Xi\) of the BGKBE~\eqref{eq:BGKBE} to a countable finite set. 
Let $\epsilon\in\Rzsp$ and $\nu$ be fixed, and define the set
\begin{align}
Q= \left\{ \bm{v}_{i} \mid i=0,1,\ldots, q-1\right\} \subseteq \Xi = \mathbb{R}^{d} 
\end{align}
which is countable and finite with \(q \coloneqq \# Q < \infty\). 
Below, we regularly denote \(Q\) as \(DdQq\) instead.
\begin{figure}[ht!]
\centering
 \includegraphics[scale=1]{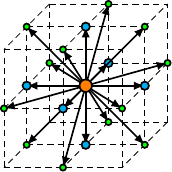}
\hspace{2em}
\includegraphics[scale=1]{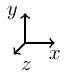}
  \caption{The \(D3Q19\) discrete velocity set.
Coloring refers to energy shells: orange, cyan, green denote zeroth, first, second order, respectively.
  }
  \label{fig:D3Q19}
\end{figure}
Unless stated otherwise, in the present work we use \( Q = D3Q19\) which is depicted in Figure~\ref{fig:D3Q19} and defined as follows. 
\begin{definition}
The \(D3Q19\) stencil is defined through its elements  
\begin{align}
\mathbb{R}^{3} \ni \bm{v}_{i} = \frac{1}{\epsilon} \begin{cases} 
\left( 0, 0, 0 \right)^{\mathrm{T}} & \quad \text{if } i=0, \\
\left( \pm 1, 0 , 0 \right)^{\mathrm{T}}, \left( 0 , \pm 1 , 0 \right)^{\mathrm{T}}, \left( 0, 0 , \pm 1 \right)^{\mathrm{T}}  & \quad \text{if } i=1,2,\ldots,6 , \\
\left( 0, \pm 1 , \pm 1 \right)^{\mathrm{T}}, \left( \pm 1, \pm 1 , 0 \right)^{\mathrm{T}}, \left( \pm 1, 0 , \pm 1 \right)^{\mathrm{T}} & \quad \text{if } i=7, 8,  \ldots, 18  ,
\end{cases} 
\end{align} 
distributed on three energy shells \cite{karlin1999perfect}.
\end{definition}

In the following, we assume that a solution $f^\epsilon$ to \eqref{eq: lbm discrete velocity 3} exists and is integrable to well-defined hydrodynamic moments $n_{f^\epsilon}$ and $\bm{u}_{f^h}$. 
The dependence of \(f^{\epsilon}\) on \(\epsilon\) transfers to its moments, such that $n_{f^\epsilon}, \bm{u}_{f^\epsilon} \in \mathcal{O}(1)$ is not assumed, but holds by construction instead. 
Taylor expanding $M_{f^\epsilon}^{\mathrm{eq}} ( t, \bm{x}, \bm{v}_{i} )$ in $\epsilon$ such that in $I\times\Omega\times Q$ a separation at order two yields 
\begin{align}\label{eq:discVeloEquilibrium}
 M_{f^{\epsilon}}^{\mathrm{eq}} ( t, \bm{x}, \bm{v}_{i} )
  &= 
  \frac{n_{f^\epsilon} \epsilon^d}{\left( \frac{2}{3} \pi \right)^{\frac{d}{2}}} \exp\left( -\frac{3}{2} \widetilde{\bm{v}}_i^2 \right) \exp\left( 3 \epsilon \widetilde{\bm{v}}_i \!\cdot\! \bm{u}_{f^\epsilon} - \frac{3}{2} \epsilon^2 \bm{u}^2_{f^\epsilon} \right) \nonumber \\
  &= 
  \underbrace{\frac{n_{f^\epsilon} \epsilon^d}{ \left( \frac{2}{3} \pi \right)^{\frac{d}{2}}} \exp\left(- \frac{3}{2}{\widetilde{\bm{v}}_i}^2\right)
  \left[ 1 + 3 \epsilon {\widetilde{\bm{v}}_i} \!\cdot\! \bm{u}_{f^\epsilon} - \frac{3}{2} \epsilon^2 \bm{u}^2_{{f^\epsilon}} + \frac 9 2 \epsilon^2 \left( {\widetilde{\bm{v}}_i} \!\cdot\! \bm{u}_{f^\epsilon} \right)^2\right]}_{\eqqcolon~\widetilde{M}_{f^\epsilon}^{\mathrm{eq}}} \\
  &\quad + 
  R^{(0)}_{t,\bm{x},\bm{v}_i}, \nonumber
\end{align}
where $R^{(0)}_{t,\bm{x},\bm{v}_i} \in \mathcal{O}(\epsilon^{d+3})$ defines the a remainder term for any $(t,\bm{x},\bm{v}_i) \in I\times\Omega\times Q$. 
Where unambiguous, we drop the corresponding indexes below. 
Note that the prefactorization ${\widetilde{\bm{v}}_i} := \epsilon\bm{v}_i $ removes the $\epsilon$-dependence.
To uphold conservation properties inherited by the evaluated Maxwellian distribution $M_{f^\epsilon}^{\mathrm{eq}}$ we introduce the weights $w_i\in\Rzsp$ for $i=0,1,...,q-1$ such that in \(I\times \Omega\) 
\begin{align}
 n_{f^{\epsilon}}  
 & = 
  \sum\limits_{i=0}^{q-1} w_i \widetilde{M}_{f^\epsilon}^{\mathrm{eq}}, \\
 n_{f^{\epsilon}} \bm{u}_{f^{\epsilon}} 
 & = 
 \sum\limits_{i=0}^{q-1} w_i \bm{v}_i \widetilde{M}_{f^\epsilon}^{\mathrm{eq}} .
\end{align}
Using Gauss--Hermite quadrature the weights $w_i$ for $D3Q19$ are deduced as 
\begin{align}\label{eq:weights}
w_i = w \begin{cases} 
    \frac{1}{3}  & \quad \text{if } i=0, \\
    \frac{1}{18}  & \quad \text{if } i=1,2,\ldots, 6 ,\\
	\frac{1}{36} & \quad \text{if } i=7,8,\ldots, 18 ,
	\end{cases}
\end{align}
where  
\begin{align}\label{eq:weightFunction}
w \left( \tilde{\bm{v}}_{i} \right) \coloneqq \left( \frac 2 3 \pi \right) ^{\frac{d}{2}} \epsilon^{-d} \exp\left(\frac{3}{2}{\widetilde{\bm{v}}_i}^2\right).
\end{align}
\begin{lemma}\label{lem:discVeloMomentApprox}
With the above, the integral moments of $f^\epsilon$ are approximated in $I\times\Omega$ by sums over the discrete velocities. 
In particular, for zeroth and first order this gives
\begin{align} 
   n_{f^\epsilon} - \sum\limits_{i=0}^{q-1} w_i f^\epsilon  &  \in \mathcal{O}\left(\epsilon^2 \right), \label{eq: lbm discrete velocity 9} \\
   \bm{u}_{f^\epsilon} - \frac{\sum\limits_{i=0}^{q-1} w_i \bm{v}_i {f^\epsilon}}{\sum\limits_{i=0}^{q-1} w_i f^\epsilon} &\in \mathcal{O}\left(\epsilon \right) .\label{eq: lbm discrete velocity 9a}
\end{align}
\end{lemma}

\begin{proof}
In \(I\times \Omega\) we can classify  
\begin{align} 
 \int\limits_{\Rz^d} f^\epsilon \,\mathrm{d}\bm{v} = n_{f^\epsilon} = \sum_{i=0}^{q-1} w_i \widetilde{M}_{f^\epsilon}^{\mathrm{eq}}
                                     &= \sum_{i=0}^{q-1} w_i M_{f^\epsilon}^{\mathrm{eq}} + R^{(1)}_{t,\bm{x}} \nonumber \\
                                     &= \sum_{i=0}^{q-1} w_i \left( f^\epsilon + 3\nu \epsilon^2 \matD f^\epsilon \right) + R^{(1)}_{t,\bm{x}} \nonumber \\
                                     &= \sum_{i=0}^{q-1} w_i f^\epsilon + R^{(2)}_{t,\bm{x}} 
\end{align}
and
\begin{align} 
 \int\limits_{\Rz^d} \bm{v} f^\epsilon \,\mathrm{d}\bm{v} = n_{f^\epsilon} \bm{u}_{f^\epsilon} 
                                     = \sum_{i=0}^{q-1} w_i \bm{v}_i \widetilde{M}_{f^\epsilon}^{\mathrm{eq}}                   &= \sum_{i=0}^{q-1} w_i \bm{v}_i M_{f^\epsilon}^{\mathrm{eq}} + R^{(3)}_{t,\bm{x}} \nonumber \\
                                     &= \sum_{i=0}^{q-1} w_i \bm{v}_i \left(f^\epsilon + 3\nu \epsilon^2 \matD f^\epsilon \right) + R^{(3)}_{t,\bm{x}} \nonumber \\
                                                          &= \sum_{i=0}^{q-1} \left(w_i \bm{v}_i f^\epsilon \right) + R^{(4)}_{t,\bm{x}} ,
\end{align}
respectively.
We categorize the remainder terms \(R_{\cdot, \cdot,\cdot}^{(\cdot)}\) as follows. 
For all $i=0,1,...,q-1$, the weight $w_i$ defined in \eqref{eq:weights} is seen as a function of $\epsilon$ in the sense of \eqref{eq:weightFunction} such that $w_i \in \mathcal{O}(\epsilon^{-d})$. 
Thus, the product $w_i(\matDil) f^\epsilon(t,\bm{x},\bm{v}_i)$ is also a function of $\epsilon$ and hence in $\mathcal{O}(1)$ for all $(t,\bm{x},\bm{v}_i) \in I\times\Omega\times Q$. 
Further, recall from \eqref{eq:discVeloEquilibrium} that $M_{f^\epsilon}^{\mathrm{eq}}$ is approximated by $\widetilde{M}_{f^\epsilon}^{\mathrm{eq}}$ with an error in $\mathcal{O}(\epsilon^{3+d})$ for all $(t,\bm{x},\bm{v}_i) \in I\times\Omega\times Q$. 
Henceforth, with $\bm{v}_i \in \mathcal{O}(\epsilon^{-1})$ by construction, we obtain for all $(t,\bm{x}) \in I\times\Omega$ that
\begin{itemize}
\item $R^{(1)}_{t,\bm{x}} \in \mathcal{O}(\epsilon^3)$,  
\item $R^{(2)}_{t,\bm{x}} \in \mathcal{O}(\epsilon^2)$,  
\item $R^{(3)}_{t,\bm{x}} \in \mathcal{O}(\epsilon^2)$, and  
\item $R^{(4)}_{t,\bm{x}} \in \mathcal{O}(\epsilon)$ 
\end{itemize}
which completes the proof.
\end{proof}
\begin{definition}\label{def:discVeloQuant}
For $i=0,1,...,q-1$ and in $I\times\Omega$ we define
\begin{align}
   f_i^\epsilon (t, \bm{x})      &\coloneqq  w_i f^\epsilon \left(t,\bm{x},\bm{v}_i\right) , \label{eq:discreteVeloPDF}\\
   n_{\bm{f}^{\epsilon}}(t,\bm{x})  &\coloneqq \sum_{i=0}^{q-1} f_i^\epsilon(t,\bm{x}) , \\
   \bm{u}_{\bm{f}^{\epsilon}}(t,\bm{x})  &\coloneqq \frac{1}{ n_{\bm{f}^{\epsilon}}(t,\bm{x})} \sum_{i=0}^{q-1} \bm{v}_i f_i^\epsilon(t,\bm{x}) , \\
   \overline{M}_{\bm{f}^{\epsilon},i}^{\mathrm{eq}}(t,\bm{x})   &\coloneqq \left\{ \frac{w_i}{w} { n_{\bm{f}^{\epsilon}}} \left[ 1 +   3 \epsilon^2 \bm{v}_i\!\cdot\! \bm{u}_{\bm{f}^{\epsilon}} - \frac{3}{2} \epsilon^2 \bm{u}_{\bm{f}^{\epsilon}}^2 + \frac 9 2 \epsilon^4 \left(\bm{v}_i\!\cdot\! \bm{u}_{\bm{f}^{\epsilon}}\right)^2\right]\right\} (t,\bm{x}). \label{eq:discreteVeloMax}
\end{align}
\end{definition}
\begin{definition}
Multiplication $w_i \times$\eqref{eq: lbm discrete velocity 3} and injection of Definition~\ref{def:discVeloQuant}, we obtain the discrete velocity BGKBE as a system of \(q\) equations 
\begin{align}\label{eq: lbm discrete velocity 10}
 \matD f_i^\epsilon &= - \frac 1 {3\nu \epsilon^2} \left( f_i^\epsilon-\overline{M}_{\bm{f}^{\epsilon},i}^{\mathrm{eq}} \right) \quad \text{in }I\times\Omega , 
\end{align}
for \(i=0,1,\ldots, q-1\). 
The upheld parametrization with \(\epsilon\) generates the family of discrete velocity BGKBE 
\begin{align} \label{eq: lbm discrete velocity 13}
\mathcal{F\!G} := \left( \matD f_i^\epsilon + \frac 1 {3\nu \epsilon^2} \left( f_i^\epsilon-\overline{M}_{\bm{f}^{\epsilon},i}^{\mathrm{eq}} \right) = 0 \quad\text{in } I\times\Omega\times Q \right)_{\epsilon>0} .
\end{align}
\end{definition}
\begin{theorem} \label{theo: lbm discrete velocity}
Suppose that for given $\epsilon,\nu \in\Rzsp$, $f^\epsilon$ is a weak solution of the BGKBE \eqref{eq: lbm discrete velocity 3} with well-defined integral moments $n_{f^\epsilon}$ and $\bm{u}_{f^\epsilon}$ and that for $n_{f^\epsilon}$, $\bm{u}_{f^\epsilon}$, $w_i\matD f^\epsilon$ understood as functions of $\epsilon$ holds
\begin{align} \label{eq: lbm discrete velocity 11} 
 n_{f^\epsilon}         &\in \mathcal{O}(1) \quad \text{in }I\times\Omega , \\ 
 \bm{u}_{f^\epsilon}    &\in \mathcal{O}(1) \quad \text{in }I\times\Omega , \\ 
 w_i \matD f^\epsilon   &\in \mathcal{O}(1) \quad \text{in }I\times\Omega\times Q . 
\end{align}
Then, the family $\mathcal{F\!G}$ of discrete velocity BGKBEs~\eqref{eq: lbm discrete velocity 13} is limit consistent of order two to the family $\mathcal{F}$ of BGKBEs~\eqref{eq: lbm discrete velocity 4} in \(I \times \Omega \times Q\). 
\end{theorem}
\begin{proof}
As a direct consequence of measuring the remainder terms which gives \eqref{eq: lbm discrete velocity 9} and \eqref{eq: lbm discrete velocity 9a} in Lemma~\ref{lem:discVeloMomentApprox}, we obtain the overall truncation error
\begin{align} \label{eq: lbm discrete velocity 12}
 \epsilon^2\matD f^\epsilon_i + \frac 1 {3\nu} \left( f^\epsilon_i-\overline{M}_{\bm{f}^{\epsilon},i}^{\mathrm{eq}}\right) &\in \mathcal{O}\left(\epsilon^2\right) \quad \text{in } I\times\Omega\times Q 
\end{align}
for \(i=0,1,\dots, q-1\).
Thus, the conditions in Definition~\ref{def:limitConsistency} are verified at order \(k=2\).
\end{proof}
\begin{proposition}
Under the premise of linear stability of the velocity discretization, solutions of the family of discrete velocity BGKBEs \(\mathcal{F\!G}\) \eqref{eq: lbm discrete velocity 13} converge weakly to solutions of the incompressible NSE \eqref{eq:incNSE} with second-order limit consistency.  
\end{proposition}
\begin{proof}
With Lemma~\ref{lem:weakConvWithLC}, we can unfold the convergence of \(\mathcal{F\!G}\) with second order consistent discretization, i.e.\  
\begin{align}
\mathcal{F\!G} \xrightarrow[\mathcal{O}\left(\epsilon^{2}\right)]{\epsilon \searrow 0} \mathrm{NSE} \equiv ( \mathcal{F} \xrightharpoonup[\mathcal{O}(\epsilon^{2})]{\epsilon \searrow 0} \mathrm{NSE} ) \circ (\mathcal{F\!G} \xrightarrow[\mathcal{O}\left(\epsilon^{2}\right)]{\epsilon\searrow 0} \mathcal{F}) .
\end{align}
\end{proof}

\subsection{Discretization of space and time}

Through completing the discretization, the parameter \(\epsilon\) is glued to the space-time grid. 
This procedure resembles the typical connection of relaxation and discretization limit in LBM \cite{simonis2020relaxation} to uphold the consistency to the initially targeted NSE \eqref{eq:incNSE}. 
In the following, we provide a consistent discretization in exactly that sense. 

Let \(\Omega\) be uniformly discretized by a Cartesian grid \(\Omega_{h}\) with \(N+1\) nodes \(x\) per dimensional direction, where \(h\) denotes the grid parameter (as abstracted in Figure~\ref{fig:limitConsAbs}). 
For the largest cubic subdomain \(\tilde{\Omega}_{h} \subseteq \Omega_{h} \subseteq \Omega\) we define \(\triangle x = \vert \tilde{\Omega}_{h} \vert^{1/d} / N \). 
Via imposing the spatio-temporal coupling \(\triangle t \sim \triangle x^{2}\), we retain a positioning constraint \(v_{i\alpha} = \flat \triangle x /\triangle t\) on the grid nodes for \(\flat \in \left\{0,\pm 1\right\}\) (up to three shells) and define the discrete time interval \(I_{\triangle t} \coloneqq \left\{ t=t_0+k\triangle t \;\vert\; t_{0} \in I , k\in\mathbb{N} \right\} \subseteq I \). 
Here and below, the continuous \(\epsilon\)-parametrization is linked with the grid parameters by mapping \(\epsilon \mapsfrom h =  \triangle x \sim \sqrt{ \triangle t }\).
\begin{definition}
Let \(I_{h} \times \Omega_{h} \times Q\) be the discrete version of \(\mathfrak{R}\) and constructed as above. 
Assume that \(f^{h}\) denotes a solution to the BGKBE \eqref{eq: lbm discrete velocity 3}. 
Continuing from \eqref{eq:discreteVeloPDF}, for \(i = 0,1,\ldots, q-1\) and \(\chi \in \mathbb{R}\), we define the \(i\)-th population
\begin{align}\label{eq:populations}
f_{i}^{h} \left( t + \chi h^{2} \right) = f_{i}^{h} \left( t + \chi h^{2}, \cdot \right) \coloneqq f_{i}^{h} \left( t + \chi h^{2}, \bm{x} + \bm{v}_{i} \chi h \right) 
\end{align}
on the space-time cylinder \((t,\bm{x}) \in I_{h} \times \Omega_{h}\). 
\end{definition}
Below, external forces \(\bm{F} = 0\) are neglected. 
For the inclusion of forces, we refer the interested reader to classical results on forcing schemes (e.g. \cite{guo2002forcing}). 

Let $f^h$ be at least of class \(C^{3}\) with respect to \(\matDil\). 
We successively assume that \((\matDil)^{3} f_{i}^{h} \in \mathcal{O}(1)\) for \(i=0,1,\ldots ,q-1\) as functions of \(h\). 
To derive the LBE, we use three discrete points in time \((t, t+ h^{2}/2, t+ h^{2})\) (see Figure~\ref{fig:time}), where the midpoint \(t+h^{2}/2\) is a ghost node for the derivation and cancels in the final evolution equation of the scheme. 
Since the width of the stencil is mapped from the initial \(\epsilon\)-scaling, limit consistency is ensured. 
\begin{figure}[ht!]
\centering
\includegraphics[scale=1]{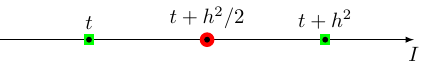}
\caption{Time steps used for the space-time discretization. 
The final LBE operates on the green squares only. 
Red points are canceled in the derivation.}
\label{fig:time}
\end{figure}

\begin{definition}\label{def:finiteDifferenceOperators}
Let \(g \colon \mathbb{R} \to \mathbb{R}, y \mapsto g(y)\) denote a \(C^{\infty}\)-function and \(a\in \mathbb{R}_{\geq 0}\) be a point on the non-negative real line. 
We define the following finite difference operations: 
\begin{enumerate}
    \item central difference 
    \begin{align}\label{eq:centralFD}
        g^{\prime}\left( y\right) = \frac{1}{a} \left( g\left( y + \frac{a}{2}\right) - g\left(y - \frac{a}{2}\right) \right) + \mathcal{O}\left( a^{2}\right) , 
    \end{align}
    \item forward difference 
    \begin{align}\label{eq:forwardFD}
        g^{\prime} \left(y\right) = \frac{1}{a} \left( g\left(y+a\right) - g\left( y\right) \right)+ \mathcal{O} \left( a \right) , 
    \end{align}
    \item Taylor's theorem 
    \begin{align}\label{eq:taylorTheorem}
        g\left( y\right) = g\left( a \right) + g^{\prime}\left( a \right) \left( y - a\right) + \mathcal{O}\left( \left\vert y - a\right\vert^{2}\right) .
    \end{align}
\end{enumerate}
\end{definition}
Starting with a central difference of \(\mathrm{D}/(\mathrm{D}t)f_{i}^{h}\) at \(t+h^{2}/2\) yields
\begin{align} \label{eq: lbm discrete lbm 1}
 h^2 \matD f^{h}_{i} \left( t + \frac{1}{2}h^2 \right) = f^{h}_{i} \left( t + h^{2} \right) - f^{h}_{i} \left( t \right) + R^{(5)}_{t,\bm{x}} \quad \text{in } I_{h} \times \Omega_{h} .
\end{align}
Taylor's theorem applied to \(f_{i}^{h}\) at at \(t+h^{2}/2\) with the expansion point \(t\), and a forward difference \((\matDil) f^{h}_{i}\) at \(t\) yields
\begin{align} 
 f^{h}_{i} \left( t + \frac{1}{2} h^{2} \right) 
 &= 
 f^{h}_{i} \left( t \right) + \frac{1}{2} h^{2} \matD f^{h}_{i} \left( t \right)  + R^{(6)}_{t,\bm{x}} \nonumber \\
 &= 
 f^{h}_{i} \left( t \right) + \frac{1}{2} \left[ f^{h}_{i} \left( t + h^{2} \right) - f^{h}_{i} \left( t \right)  \right] + R^{(7)}_{t,\bm{x}} \label{eq: lbm discrete lbm 2}
\end{align}
in \(I_{h} \times \Omega_{h}\). 
Provided that \((\matDil)^{2} f_{i}^{h} \in \mathcal{O}( 1)\) and \((\matDil)^{3} f_{i}^{h} \in \mathcal{O}( 1)\), we deduce 
\begin{itemize}
\item $R^{(5)}_{t,\bm{x}} \in \mathcal{O}(h^6)$,
\item \(R_{t,\bm{x}}^{(6)} \in \mathcal{O} (h^{4})\), and
\item \(R_{t,\bm{x}}^{(7)} \in \mathcal{O} (h^{4})\) for \((t,\bm{x}) \in I_{h} \times \Omega_{h}\). 
\end{itemize}
To match the observation point for the discrete moments as functions of \(h\), we shift \(n_{\bm{f}^{h}}\) and \(\bm{u}_{\bm{f}^{h}}\) by \(h^{2}/2\), respectively in space-time. 
Recalling the diffusive limit (cf. \eqref{eq:mass} and \eqref{eq:momentum}), $(\matDil) n_{f^{h}} = 0$ and $(\matDil) \bm{u}_{f^{h}} = 0$ instantly follows. 
Via \((\matDil)\)\eqref{eq: lbm discrete velocity 9} and \((\matDil)\)\eqref{eq: lbm discrete velocity 9a}, we have for \((t,x) \in I_{h} \times \Omega_{h}\) that 
\begin{align}
\matD n_{\bm{f}^{h}} = \mathcal{O}(h^2), \label{eq:matDerivZeroth}\\
\matD \bm{u}_{\bm{f}^{h}} = \mathcal{O}(h), \label{eq:matDerivFirst}
\end{align}
Taylor expanding both, \eqref{eq:matDerivZeroth} and \eqref{eq:matDerivFirst} for \((t,x) \in I_{h}\times \Omega_{h}\) leads to the respective approximations 
\begin{align} 
 n_{\bm{f}^{h}} \left( t + \frac{1}{2}h^2 \right)  
&=
 n_{\bm{f}^{h}} \left( t \right) + \frac{1}{2} h^{2}  \matD n_{\bm{f}^{h}} \left( t \right)  + R^{(8)}_{t,\bm{x}} \nonumber \\
&= 
 n_{\bm{f}^{h}} \left( t \right) + R^{(9)}_{t,\bm{x}},   \label{eq:lbm discrete lbm density}\\
 \bm{u}_{\bm{f}^{h}} \left( t + \frac{1}{2} h^{2} \right) 
&= 
\bm{u}_{\bm{f}^{h}} \left( t \right) + \frac{1}{2} h^{2} \matD \bm{u}_{\bm{f}^{h}} \left( t\right)  + R^{(10)}_{t,\bm{x}}  \nonumber\\
&= 
\bm{u}_{\bm{f}^{h}} \left( t \right) + R^{(11)}_{t,\bm{x}}. \label{eq: lbm discrete lbm 2a}
\end{align}
with remainder terms $R^{(8)}_{t,\bm{x}}, R^{(9)}_{t,\bm{x}}, R^{(10)}_{t,\bm{x}} \in \mathcal{O} ( h^4 ) $ and $R^{(11)}_{t,\bm{x}} \in \mathcal{O} ( h^3 )$.
At last, we compute the lattice Maxwellian for \((t,x) \in I_{h} \times \Omega_{h}\) with Taylor's theorem of \eqref{eq:discreteVeloMax} as
\begin{align} \label{eq: lbm discrete lbm 3}
 \overline{M}^{\mathrm{eq}}_{\bm{f}^{h},i} \left( t + \frac{1}{2} h^{2}, x + v_{i} \frac{1}{2} h^{2} \right) = \overline{M}^{\mathrm{eq}}_{\bm{f}^{h},i} \left( t, x \right) + R^{(12)}_{t,\bm{x}} 
\end{align}
To specify the error we use \eqref{eq:lbm discrete lbm density} and \eqref{eq: lbm discrete lbm 2a} for \(\overline{M}_{f_{i}^{h}}^{\mathrm{eq}}\) as given in \eqref{eq: lbm discrete lbm 3}. 
Further, from the meanwhile gathered assumptions that \((\mathrm{D}/(\mathrm{D}t))^{j} f_{i}^{h} \in \mathcal{O}(1)\) for \(j=0,1,2,3\) we deduce that \((\mathrm{D}/(\mathrm{D}t))\overline{M}_{f_{i}^{h}}^{\mathrm{eq}} \in \mathcal{O}(1)\). 
Hence, being prefactored by \(h^{2}\) in the leading order, the remainder $R^{(12)}_{t, \bm{x}} \in \mathcal{O}(h^2)$ for all $(t,\bm{x}) \in I_h \times \Omega_{h}$. 

\begin{definition}
Based on the above discretizations \eqref{eq:populations}, \eqref{eq: lbm discrete lbm 1}, \eqref{eq: lbm discrete lbm 2}, and \eqref{eq: lbm discrete lbm 3}, we construct the lattice Boltzmann equation (LBE) for a space-time cylinder \(I_{h} \times \Omega_{h}\) and \(i=0,1,\ldots , q-1\) discrete velocities via reordering the terms of \eqref{eq: lbm discrete velocity 10}\(\vert_{(t+(1/2) h^{2}, x+\bm{v}_{i} (1/2) h^{2})}\) to 
\begin{align} \label{eq: lbm discrete lbm 4}
 f_{i}^{h} \left( t + h^{2} \right) - f^{h}_{i} \left(  t \right) = -\frac{1}{3 \nu + \frac{1}{2}} \left[ f_{i}^{h} \left( t \right) - \overline{M}^{\mathrm{eq}}_{\bm{f}^{h},i} \left( t \right) \right] . 
\end{align}
Consequently, the family of LBEs reads 
\begin{align} \label{eq: lbm discrete lbm 6}
 \mathcal{G}:= \left( f_i^{h} \left( t+h^{2} \right) - f_{i}^{h} \left( t \right) + \frac{1}{3\nu + \frac{1}{2}} \left[ f_{i}^{h} \left( t \right) - \overline{M}^{\mathrm{eq}}_{f_{i}^{h}} \left( t \right) \right] = 0 \quad \text{in } I_{h} \times \Omega_{h} \times Q \right)_{h>0} .
\end{align}
\end{definition}

\subsection{Consistent lattice Boltzmann equations}

\begin{theorem} \label{theo: lbm discrete lbm}
Suppose that for given $h,\nu \in \mathbb{R}_{>0}$, $f^h$ is a weak solution of the BGKBE \eqref{eq: lbm discrete velocity 3} with moments $n_{f^{h}}$ and $\bm{u}_{f^{h}}$. 
Further, let $n_{f^{h}}$, $\bm{u}_{f^{h}}$, $(\matDil)^{j} f_{i}^{h}$ understood as functions of $h$ fulfill that
\begin{align} \label{eq: lbm discrete lbm 5}
 n_{f^{h}} \in \mathcal{O}(1) \quad & \text{in } I_h \times \Omega_h , \\ 
 \bm{u}_{f^{h}} \in \mathcal{O}(1) \quad & \text{in } I_h \times \Omega_h , \\ 
\left(\matD\right)^{j} f_{i}^{h} \in \mathcal{O}(1) \quad & \text{in } I_h \times \Omega_h , \quad\text{for } j= 0,1,2,3, 
\end{align}
for \(i=0,1,\ldots ,q-1\).
Then, the family $\mathcal{G}$ of LBEs is limit consistent of order two to the family $\mathcal{F}$ of BGKBEs in $I_h\times\Omega_h\times Q$.
\end{theorem}
\begin{proof}
Let $f^h$ be a solution of the BGKBE \eqref{eq: lbm discrete velocity 3} as specified above. 
Then, Theorem~\ref{theo: lbm discrete velocity} dictates the truncation error for discretizing \(\Xi\) down to \(Q\). 
In particular, for all \(i=0,1,\ldots, q-1\)
\begin{align}  
\mathcal{E}(h) \equiv \frac{6\nu}{6\nu+1} \left[ h^2\matD f^h_i + \frac{1}{3\nu} \left( f^{h}_{i}-\overline{M}_{\bm{f}^{h},i}^{\mathrm{eq}}\right)\right] \in \mathcal{O}\left( h^{2} \right) \quad \text{in } I\times\Omega\times Q .
\end{align}
Considering  $\mathcal{E}(h) \vert_{(t+(1/2) h^2, x + \bm{v}_{i}(1/2)h^{2})}$ and substituting the Taylor expansions and finite differences constructed in \eqref{eq: lbm discrete lbm 1}, \eqref{eq: lbm discrete lbm 2}, and \eqref{eq: lbm discrete lbm 3}, unfolds that
\begin{align}
 \frac{6\nu}{6\nu+1}  \biggl\{ h^2 &\matD f^h_i \left( t+\frac{1}{2} h^2\right) + \frac{1}{3\nu} \left[ f^h_i \left( t+ \frac{1}{2} h^2 \right) - \overline{M}_{\bm{f}^{h},i}^{\mathrm{eq}} \left( t+ \frac{1}{2} h^2 \right) \right] \biggr\} \nonumber \\
 & \approx \frac{6\nu}{6\nu+1} \biggl( f^h_i \left(t+h^2\right) - f^h_i(t) + \frac 1 {3\nu} \left\{ f^h_i(t) + \frac{1}{2} \left[ f^h_i \left(t+h^2 \right) - f^h_i (t) \right] - \overline{M}_{\bm{f}^{h},i}^{\mathrm{eq}} (t) \right\} \biggr) \nonumber \\
 &= f^h_i \left( t+h^2 \right) - f^h_i(t) + \frac{1}{3\nu+\frac{1}{2}}  \left[ f^h_i (t) -\overline{M}_{\bm{f}^{h},i}^{\mathrm{eq}}(t) \right] \in \mathcal{O}\left(h^2\right)
\end{align}
for $i=0,1,...,q-1$ and for all $t\in I_h$. 
The truncation error order is obtained from inserting the respective error terms \(R_{t,\bm{x},\bm{v}}^{(j)}\) for \(j=0,1,\ldots,12\). 
\end{proof}
\begin{proposition}
Assuming stability in terms of Definition~\ref{def:stabilityLeveque}, solutions to the family of LBEs \(\mathcal{G}\) converge weakly to solutions of the incompressible NSE \eqref{eq:incNSE} with second order limit-consistency and a truncation error \(\mathcal{O}(\triangle x^{2})\).  
\end{proposition}
\begin{proof}
The weak convergence of \(\mathcal{G}\) with second order consistent discretization with \(\epsilon \mapsfrom h = \triangle x \sim \sqrt{\triangle t}\) is determined via the concatenation of weak solution functors 
\begin{align}
\mathcal{G} \xrightarrow[\mathcal{O}\left( h^{2} \right)]{h \searrow 0} \mathrm{NSE} \equiv ( \mathcal{F} \xrightharpoonup[\mathcal{O}(\epsilon^{2})]{\epsilon \searrow 0} \mathrm{NSE} ) \circ (\mathcal{G} \xrightarrow[\mathcal{O}\left( h^{2} \right)]{h\searrow 0} \mathcal{F})
\end{align}
according to Lemma~\ref{lem:weakConvWithLC}. 
\end{proof}
\begin{remark}
Here, we effectively derive the LBE by chaining Taylor expansions and finite differences with the inclusion of a shift by half a space-time grid spacing. 
As suggested in the literature \cite{ubertini2010three,dellar2013interpretation}, the latter substitutes the implicity, resembles a Crank--Nicolson scheme, and unfolds the LBE as a finite difference discretization of the BGKBE. 
In addition, we couple the discretization to a sequencing parameter which is responsible for the weak convergence of mesoscopic solutions to solutions of the macroscopic target PDE. 
This finding aligns with recent works \cite{bellotti2022finite,bellotti2022rigorous} which rigorously express lattice Boltzmann methods in terms of finite differences for the macroscopic variables. 
\end{remark}

\section{Conclusion}\label{sec:conclusion}

The notion of limit consistency is introduced for the constructive discretization towards an LBM for a given PDE. 
In particular, a family of LBEs $\mathcal{G}$ is consistently derived from a family of BGKBEs $\mathcal{F}$ which is known to rigorously limit towards a given partial differential IVP in the sense of weak solutions. 
Preparatively, the discrete velocity BGKBE $\mathcal{F\!G}$ is constructed by replacing the velocity space with a discrete subset. 
When completed, we discretize continuous space and time via chaining several finite difference approximations which leads to $\mathcal{G}$. 
In the present example, all of the discrete operations retain the rigorously proven diffusive limit of the BGKBE towards the NSE on the continuous level. 

It is to be noted that the term consistency in the derivation process refers to the level-by-level preservation of the \textit{a priori} rigorously proven limit. 
As such, our approach presents itself as mathematically rigorous discretization procedure, which produces a numerical scheme with a proven limit towards the target equation under diffusive scaling. 
Albeit the incompressible Navier--Stokes equations (NSE) are used as an example target system, the current approach is extendable to any other PDE. 
In addition, the notion of limit consistency is flexible in the sense that it can be applied to relaxation systems without the knowledge of an underlying kinetic equation system limiting to the targeted PDE. 
Hence, the sole constraint remaining, is the existence of an \textit{a priori} known continuous moment limit which rigorously dictates the relaxation process. 

Further, we provide a successive discretization by nesting conventional Taylor expansions and finite differences. 
The measuring of the truncation errors at all levels within the derivation enables to track the discretization state of the equations. 
Parametrizing the latter in advance allows to predetermine the link of discrete and relaxation parameters, which is primarily relevant to uphold the path towards the targeted PDE. 
The unfolding of LBM as a chain of finite differences and Taylor expansions under the relaxation constraint matches with results in the literature \cite{junk2001finite,bellotti2022finite} and thus validates the present work. 
Although, similarities to well-established references \cite{junk2000discretizations,junk2005asymptotic} are present, it is to be stressed that the purpose of limit consistency is to touch base with the previous construction steps \cite{simonis2020relaxation,simonis2022constructing} and enable a generic top-down design of LBM. 
Future studies as prepared in \cite{simonis2023pde} will complete the coherent top-down procedure to derive LBMs for large classes of PDEs via assembling the perturbative construction of generic relaxation systems proposed in \cite{simonis2020relaxation,simonis2022constructing} with the here derived consistent discretization. 

\textbf{\emph{Funding:}}
This work was partially funded by the "Deutsche Forschungsgemeinschaft" (DFG, German
Research Foundation) {\textendash} project 382064892, KR 4259/8-2.

\textbf{\emph{Author contribution statement:}}
\textbf{S.~Simonis}: 
Conceptualization, 
Methodology, 
Formal analysis, 
Validation, 
Investigation, 
Visualization, 
Writing - Original draft, 
Project administration, 
Resources; 
\textbf{M.~J.~Krause}: 
Conceptualization, 
Writing - Review {\&} Editing, 
Supervision, 
Funding acquisition. 
All authors read and approved the final manuscript.

%

\end{document}